\documentclass[aos,preprint]{imsart}

\startlocaldefs

\RequirePackage[OT1]{fontenc}
\RequirePackage{amsthm,amsmath}
\RequirePackage[numbers]{natbib}
\usepackage{xr-hyper}
\RequirePackage[colorlinks,citecolor=blue,urlcolor=blue]{hyperref}

\usepackage{etoolbox}
\providetoggle{includeSupp}
\settoggle{includeSupp}{true}
\usepackage{enumitem}

\usepackage{algorithm2e}
\SetKwInOut{Parameter}{parameter}

\usepackage{amsmath}
\usepackage{amssymb}
\usepackage{amsthm}
\usepackage{bbm}

\usepackage{etex, etoolbox}
\usepackage{thmtools}
\usepackage{environ}

\usepackage{tkz-graph} 
\usetikzlibrary{calc} 
\usetikzlibrary{positioning} 
\usetikzlibrary{shapes.geometric}%

\tikzstyle{VertexStyle} = [shape = ellipse,
	minimum width    = 6ex,%
	draw]
\tikzstyle{EdgeStyle}   = [->,>=stealth']      
\newcommand{\G}{\mathcal{G}}
\newcommand{\I}{\mathcal{I}}
\newcommand{\ind}[3]{\langle #1, #2 \mid #3 \rangle}
\newcommand{\starrightarrow}{\ *\!\!\rightarrow}
\newcommand{\starleftarrow}{\leftarrow\!\! * \ }
\newcommand{\notstarrightarrow}{\ *\!\!\not\rightarrow}

\newcommand{\dsepG}[4]{#1 \perp_d #2 \mid #3\ [#4]}

\newcommand{\deltasepG}[4]{#1 \perp_\delta #2 \mid #3\ [#4]}
\newcommand{\msep}[3]{#1 \perp_m #2 \mid #3}
\newcommand{\msepG}[4]{#1 \perp_m #2 \mid #3\ [#4]}
\newcommand{\musep}[3]{#1 \perp_\mu #2 \mid #3}
\newcommand{\musepG}[4]{#1 \perp_\mu #2 \mid #3\ [#4]}

\newcommand{\disjU}{\mathbin{\dot{\cup}}}
\newcommand*\circled[1]{\tikz[baseline=(char.base)]{
            \node[shape=circle,draw,inner sep=1pt] (char) {#1};}}

\endlocaldefs

\numberwithin{equation}{section}
\theoremstyle{plain}
\theoremstyle{remark}
\newtheorem{thm}{Theorem}[section]

\newtheorem{cor}[thm]{Corollary}
\newtheorem{defn}[thm]{Definition}
\newtheorem{exmp}[thm]{Example}
\newtheorem{lem}[thm]{Lemma}
\newtheorem{prop}[thm]{Proposition}

\begin{document}
	
	\begin{frontmatter}
		
		\title{Markov equivalence of marginalized local independence graphs}
		
		\runtitle{Marginalized local independence graphs}
		
		\begin{aug}
			
			\author{\fnms{S{\o}ren Wengel}
			\snm{Mogensen}\ead[label=e1]{swengel@math.ku.dk}}
			\and
			\author{\fnms{Niels Richard} 
			\snm{Hansen}\ead[label=e2]{niels.r.hansen@math.ku.dk}}
			\address[a]{Department of Mathematical Sciences \\ University of 
			Copenhagen \\
				Universitetsparken 5 \\
				DK-2100 Copenhagen Ø \\
				Denmark \\
				\printead{e1}}
			\address[b]{Department of Mathematical Sciences \\ University of 
			Copenhagen \\
				Universitetsparken 5 \\
				DK-2100 Copenhagen Ø \\
				Denmark \\
				\printead{e2}}
			
			\runauthor{S. W. Mogensen \& N. R. Hansen}
			
			\affiliation{University of Copenhagen}
			
			\thankstext{}{This work was supported by VILLUM FONDEN (grant 
			13358).}
		
			\setattribute{journal}{name}{}
		\end{aug}
		
		
		
		\begin{abstract}
Symmetric independence relations are often stu\-di\-ed using
graphical representations. Ancestral graphs or acyclic directed mixed graphs 
with
$m$-separation provide classes of symmetric graphical independence models that 
are closed
under marginalization. Asymmetric independence relations 
appear naturally for multivariate stochastic processes, for instance
in terms of local independence. However, no
class of graphs representing such asym\-me\-tric independence relations, which 
is also closed under marginalization, has been developed. We develop the theory 
of
directed mixed graphs with $\mu$-separation and show that this provides a
graphical independence model class which is closed under marginalization and 
which gene\-ra\-lizes previously
considered graphical representations of local independence.

Several graphs may encode the same set of independence relations and this means 
that in many cases only an equivalence class of graphs can be identified from 
observational 
data. For statistical applications, it is therefore pivotal to characterize 
graphs that
induce the same independence relations. Our main result is that for directed 
mixed graphs
with $\mu$-separation each equivalence class contains a maximal
element which can be constructed from the independence relations
alone. Moreover, we introduce the directed mixed equivalence graph as
the maximal graph with dashed and solid edges. This graph encodes all
information about the edges that is identifiable from the independence
relations, and furthermore it can be computed efficiently from the maximal 
graph.
		\end{abstract}

	\end{frontmatter}

\section{Introduction}

Graphs have long been used as a formal tool for reasoning with
independence models. Most work has been concerned with symmetric
independence models arising from standard probabilistic independence
for discrete or real-valued random variables.  However, when working
with dynamical processes it is useful to have a notion of independence
that can distinguish explicitly between the present and the past, and
this is a key motivation for considering local independence.

The notion of local independence was introduced for composable Markov processes 
by Schweder \cite{schweder1970} who also gave examples of graphs describing 
local 
independence 
structures. Aalen \cite{aalen1987} discussed how one could extend the 
definition of 
local independence in the broad class of 
semi-martingales using the Doob-Meyer decomposition. 
Several authors have since then used graphs to represent 
local 
independence 
structures of multivariate stochastic process models -- in particular for 
point process models, see e.g. \cite{aalen2012, didelez2000, didelez2006, 
Didelez2008, roysland2012}. 
Local 
independence takes a dynamical point of view in the sense that it 
evaluates the dependence of the present on the past. This provides a natural 
link to statistical causality as cause must necessarily precede 
effect \cite{aalen1987, aalen1980, meek2014, schweder1970}. Furthermore, recent 
work argues that for some applications it can be important to consider
continuous-time
models, rather than only cross-sectional models, when trying to infer causal 
effects 
\cite{aalen2016}.

Local independence for point processes has been applied for data ana\-ly\-sis, 
see 
e.g. 
\cite{aalen1980, jensen1986, xu2016}, but in applications a direct causal
interpretation may be invalid if only certain dynamical
processes are observed while other processes of 
the system under study are unobserved. Allowing for such latent
processes is important for valid causal inference, and this motivates our study
of representations of marginalized local 
independence graphs.

Graphical representations of independence models have also
been stu\-died for time series \cite{eichler2012, eichler2013, 
eichler2007,eichler2010}. In the time series context -- using the
notion of Granger causality -- Eichler \cite{eichler2013} gave an algorithm for
learning a gra\-phi\-cal representation of local independence. However,
the equivalence class of graphs that yield the same local
independences was not identified, and thus the learned graph does not
have any clear causal interpretation. Related research has been
concerned with inferring the graph structure from subsampled time
series, but under the
assumption of no latent processes, see e.g. \cite{danks2013, danks2016}.

In this paper, we 
give a 
formal, graphical framework for handling
the pre\-sen\-ce of unobserved processes and extend the work on graphical 
representations of local 
independence models by formalizing marginalization and giving results on the 
equivalence classes of such graphical representations. The graphical framework 
that we propose is a generalization of that of Didelez \cite{didelez2000, 
didelez2006, Didelez2008}. This development is 
analogous to work on 
margina\-li\-za\-tions of graphical mo\-dels using directed acyclic graphs, 
DAGs. 
Starting 
from a DAG, one can find graphs (e.g. maximal ancestral graphs or acyclic 
directed mixed graphs) that encode marginal independence models 
\cite{cox1996, evans2016, evans2014, koster1999, richardson2002, 
richardson2017, sadeghi2013,  spirtes1997}. One can then 
characterize the equivalence class of graphs that yield the same
independence model \cite{ali2009, zhao2005} -- the so-called Markov equivalent 
graphs -- and construct learning 
algorithms to find such an
equivalence class from data. The purpose of this paper is to develop
the necessary theoretical foundation for learning local independence
graphs by developing a precise characterization of the learnable object: the
class of Markov equivalent graphs. 

The paper is structured as follows: in Section \ref{sec:graph} we discuss 
abstract independence models, relevant
graph-theoretical concepts, and the notion of local
independence and local independence graphs. In Section \ref{sec:DMGs}
we introduce $\mu$-separation for directed mixed graphs, which will be
used to represent marginalized local independence graphs, and we
describe an algorithm to marginalize a given local independence
graph. In Sections \ref{sec:propDMGs} and \ref{sec:MEofDMGs} we
develop the theory of $\mu$-separation for directed mixed graphs
further, and we discuss, in particular, Markov equivalence of such
graphs. All proofs of the main paper are given in the supplementary material. 
Sections \ref{sec:otherSeps} to
\ref{sec:proofs} are in the supplementary material.

\section{Independence models and graph theory}
\label{sec:graph}

Graphical separation criteria as well as probabilistic models
give rise to abstract conditional independence statements. Graphical modeling is
essentially about relating graphical separation to probabilistic
independence. We will consider 
both as instances of abstract {\it independence models}. 

Consider some set $\mathcal{S}$. An {\it independence model},
$\mathcal{I}$, on $\mathcal{S}$ is a set of triples $(A,B,C)$ where
$A,B,C\in \mathcal{S}$, that is,
$\mathcal{I} \subseteq \mathcal{S} \times \mathcal{S} \times
\mathcal{S}$. Mathematically, an independence model is a ternary
relation. In this paper, we will consider independence models {\it
  over} a finite set $V$ which means that
$\mathcal{S} = \mathcal{P}(V)$, the power set of $V$. In this case an
independence model $\I$ is a subset of
$\mathcal{P}(V)\times \mathcal{P}(V)\times \mathcal{P}(V)$. We will
call an element
$s \in \mathcal{P}(V)\times \mathcal{P}(V)\times \mathcal{P}(V)$ an
{\it independence statement} and write $s$ as
$\langle A,B \! \mid \! C\rangle$ for $A, B, C \subseteq V$. This
notation emphasizes that $s$ is thought of as a statement about $A$
and $B$ conditionally on $C$.

Graphical and probabilistic independence 
models have 
been studied in very general settings, though mostly under the assumption of 
symmetry of the independence model, that is,

$$
\ind{A}{B}{C} \in \I \Rightarrow \ind{B}{A}{C} \in \I,
$$

\noindent see e.g. \cite{constantinouExtCondIndep, dawidSeparoids,
lauritzen2017} and 
references 
therein. These works take an abstract axiomatic approach by describing and 
working with a number of 
properties that hold for e.g. models of conditional independence. In this 
paper, we consider independence models that do not 
satisfy the symmetry property as will become evident when we introduce the 
notion of local independence.

\subsection{Local independence}
\label{sec:locind}

We consider a real-valued, multivariate stocha\-stic process 
$$X_t = (X_t^1, X_t^2, \ldots, X_t^n),\quad t \in [0, T]$$
defined on a probability space
$(\Omega, \mathcal{F}, P)$. In this section, the process is a
continuous-time process indexed by a compact time interval. The case
of a discrete time index, corresponding to $X = (X_t)$
being a time series, is treated in Section \ref{supp:ts} in the supplementary 
material. We 
will later identify the coordinate processes of $X$ with the nodes of a graph, 
hence, both are indexed by $V = \{1,2,\ldots,n\}$. As illustrated in
Example \ref{exmp:gw} below, the index set may be chosen in a more
meaningful way for a specific application. In that example, $X_t^I
\geq 0$ is a price process, $X_t^L \in \mathbb{N}_0$ is a counting
process of events, and the remaining four processes take values in
$\{0,1\}$ indicating if an individual at a given time is a regular
user of a given substance. Figure \ref{fig:sample} shows examples
of sample paths for three individuals. 

To avoid technical difficulties, irrelevant for the present paper, we
restrict attention to right-continuous processes with coordinates of finite and
integrable variation on the interval $[0, T]$. This includes most non-explosive 
multivariate counting processes as an important
special case, but also other interesting processes such as 
piecewise-deterministic Markov processes. 

To define local independence below we need a mathematical description
of how the stochastic evolution of one coordinate process depends
infinite\-si\-mally on its own past and the past of the other
processes. To this end, let $\mathcal{F}_t^{C,0}$ denote the 
$\sigma$-algebra generated by $\{X_s^\alpha: s \leq t, \alpha \in C\}$
for $C\subseteq V$. For technical reasons we need to
enlarge this $\sigma$-algebra, and we define $\mathcal{F}_t^{C}$ to be the completion of 
$\cap_{s
  > t} \mathcal{F}_s^{C,0}$ w.r.t. $P$. Thus $(\mathcal{F}_t^{C})$ is
a right-continuous and complete filtration  
which represents the history of the processes indexed by
$C\subseteq V$ until time $t$. Figure \ref{fig:sample_filt} illustrates, in the 
context of 
Example
\ref{exmp:gw}, the filtrations $\mathcal{F}_t^{V}$,
$\mathcal{F}_t^{\{L, M, H\}}$, and $\mathcal{F}_t^{\{T, A, M, H\}}$.

For $\beta \in V$ and $C \subseteq V$ let $\Lambda^{C, \beta}$
denote an  $\mathcal{F}_t^C$-predictable process of finite and
integrable variation such that 
$$E(X^{\beta}_t \mid \mathcal{F}^C_t) -   \Lambda_t^{C, \beta}$$ 
is an $\mathcal{F}_t^C$ martingale. Such a process exists, see Section 
\ref{supp:compensators} for the
technical details, and is usually called
the compensator or the dual predictable projection of
$E(X^{\beta}_t \mid \mathcal{F}^C_t)$. It is in general unique up to
evanescence.

\begin{figure}
\center
\includegraphics[width = \textwidth]{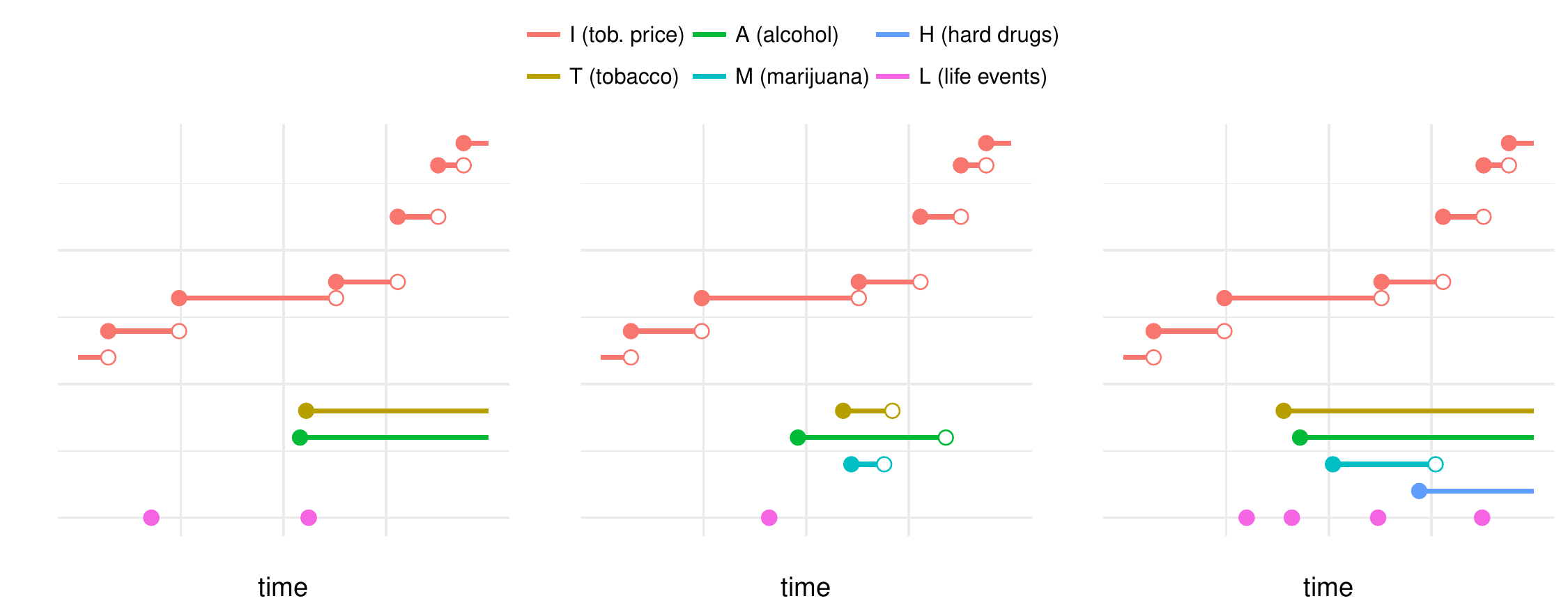}
\caption{Sample paths for three individuals of the processes
  considered in Example \ref{exmp:gw}. The price process (I) is a
  piecewise constant jump process and the life event process (L) is
  illustrated by the event times. The remaining four processes are
  illustrated by the segments of time where the individual is a
  regular user of the substance. The absence of a process, e.g. the
  hard drug process (H) in the left and middle samples, means that the
  individual never used that substance.  \label{fig:sample}}
\end{figure}

\begin{defn}[Local independence]
	Let $A,B,C \subseteq V$. We say that $X^B$ is 
	{\it locally independent} of $X^A$ given $X^C$ if there exists an 
	$\mathcal{F}_t^{C}$-predictable version of
        $\Lambda^{A \cup C, \beta}$
        for all
        $\beta \in B$. We use 
	$A \not \rightarrow B\mid C$ to denote that $X^B$ is 
	locally 
	independent of $X^A$ given $X^C$.
	\label{def:localIndep}
\end{defn}

In words, the process $X^B$ is locally independent of $X^A$ given
$X^C$ if, for each timepoint, the past up until time $t$ of $X^C$
gives us the same {\it predictable} information about $E(X^{\beta}_t \mid
\mathcal{F}^{A \cup C}_t)$ as the past of $X^{A\cup C}$ until time $t$. Note
that when $\beta \in C$, $E(X^{\beta}_t \mid \mathcal{F}^C_t) =
X_t^{\beta}$. 

Local independence was introduced by
Schweder \cite{schweder1970} for composable Mar\-kov processes and extended by 
Aalen
\cite{aalen1987}. Local independence and graphical representations
thereof were later considered by Didelez
\cite{didelez2000,didelez2006,Didelez2008} and by Aalen et al. 
\cite{aalen2012}. Didelez \cite{didelez2006} also discussed local independence 
models of composable finite Markov processes under some specific types of 
margina\-li\-zation.
Commenges and G{\'e}gout-Petit \cite{commenges2009, gegoutpetit2010} discussed 
definitions
of local independence in classes of semi-martingales. Note that
Definition \ref{def:localIndep} allows a process to be separated from
itself by some conditioning set $C$, generalizing the definition used
by e.g. Didelez \cite{Didelez2008}.

\begin{figure}
\center
\includegraphics[width = \textwidth]{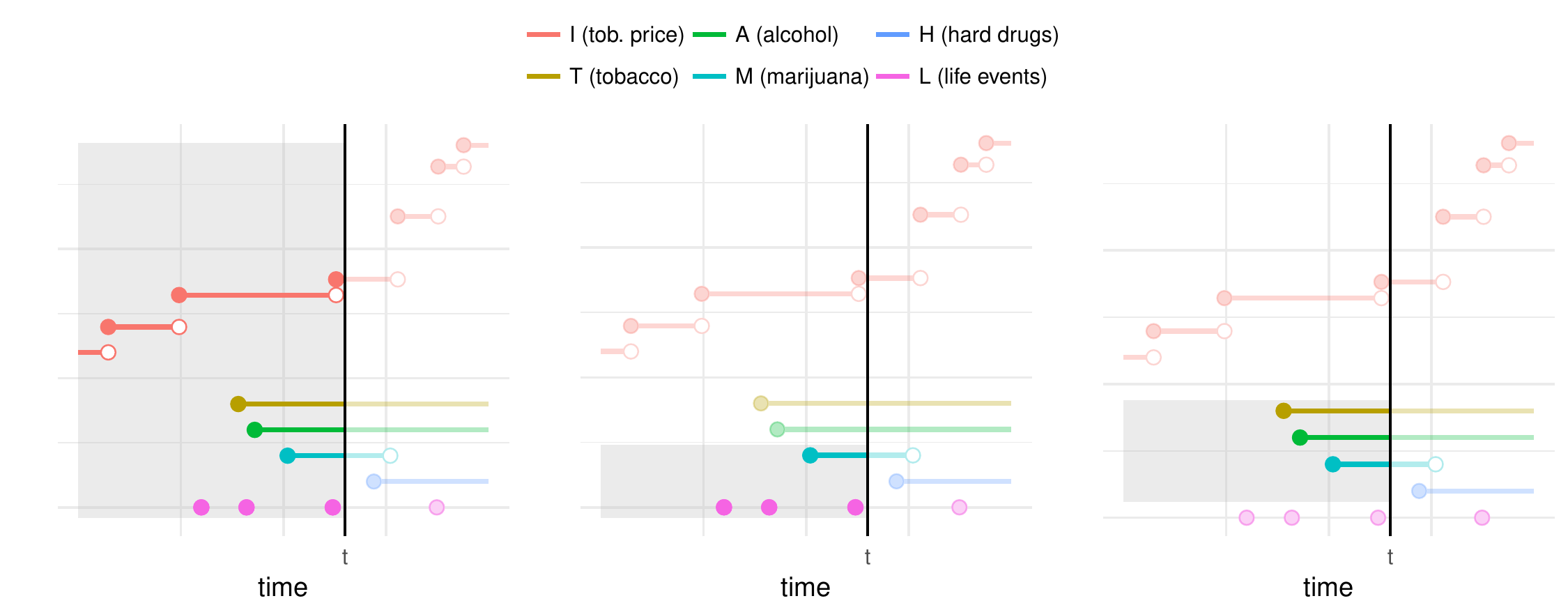}
\caption{Illustration of the past at time $t$ as captured by
  different filtrations for a single sample path of processes from
  Example \ref{exmp:gw}. The filtration $\mathcal{F}_t^{V}$ (left)
  captures the past of all processes, while $\mathcal{F}_t^{\{L, M,
    H\}}$ (middle) captures the past of  $L$, $M$, and $H$ only, and
  $\mathcal{F}_t^{\{T, A, M, H\}}$ (right) captures the past of $T$, $A$, $M$, and $H$. \label{fig:sample_filt}}
\end{figure}

Local independence defines the independence model
$$
\I = \{\ind{A}{B}{C} \mid X^B \text{ is locally independent of } X^A 
\text{ given } 
X^C \}
$$

\noindent such that the local independence statement $A\not\rightarrow
B\mid C$ is equivalent to $\langle A,B\mid C\rangle \in \I$ in the abstract notation. 
We note that the local independence model is generally not symmetric. Using 
Definition \ref{def:localIndep}, we introduce below an associated directed 
graph in 
which there is no directed edge from a node $\alpha$ to a node $\beta$ if and 
only $\beta$ is locally independent of $\alpha$ given $V\setminus \{\alpha\}$. 

\begin{defn}[Local independence graph]
For the local independence model determined by $X$, we define the {\it local 
	independence graph} to be the directed graph, $\mathcal{D}$, with nodes $V$ 
	such that 
	for $\alpha,\beta 
	\in V$
$$
\alpha \not\rightarrow_\mathcal{D} \beta \Leftrightarrow \alpha \not\rightarrow 
\beta \mid V\setminus \{\alpha\}
$$				

\noindent where $\alpha \not\rightarrow_\mathcal{D}\beta$ denotes that there is 
no directed edge from $\alpha$ to $\beta$ in the graph $\mathcal{D}$. 
\label{def:LIG}
\end{defn}	

Didelez \cite{didelez2000} gives almost the same definition of a local 
independence graph, however, in essence always assumes that there is a 
dependence of each process on its own past. See also Sections 
\ref{sec:otherSeps} and \ref{sec:examp}.

The local independence graph induces an independence model by
$\mu$-separation as defined below. The main goal of the present paper
is to provide a graphical representation of the induced independence
model for a subset of coordinate processes corresponding to the case where some
processes are unobserved. This is achieved by establishing a
correspondence, which is preserved under marginalization, between
directed mixed graphs and independence models induced via
$\mu$-separation. We emphasize that the correspondence only relates
local independence to graphs when the local independence model
satisfies the global Markov property with respect to a graph. 

The local independence model satisfies the global Markov property with
respect to the local independence graph if every $\mu$-separation in the graph 
implies a
local independence. This has been shown for point processes under some
mild regularity conditions \cite{Didelez2008} using the slightly
different notion of $\delta$-separation. Section \ref{sec:otherSeps}
discusses how $\delta$-separation is related to $\mu$-separation, and
Section \ref{sec:examp} shows how to translate the global Markov
property of \cite{Didelez2008} into our framework. Moreover, general
sufficient conditions for the global Markov property were given in
\cite{mogensen2018} covering point processes as well as certain
diffusion processes. Section \ref{supp:ts} provides, in addition, a
discussion of Markov properties in the context of time series. 

To help develop a better understanding of local independence and its
relevance for applications, we discuss an example of drug abuse progression.

\begin{exmp}[Gateway drugs]
	The theory of {\it gateway drugs} has been discussed for many years in the 
	literature on substance abuse \cite{kandel1975, vanyukov2012}. In short, 
	the theory 
	posits that the use of ``soft'' and often licit drugs precedes (and 
	possibly 
	leads to) later 
	use of ``hard'' drugs.	Alcohol, tobacco, and marijuana have all been 
		discussed as candidate gateway drugs to ``harder'' drugs such as heroin.

		\begin{figure}
			\resizebox{0.3\textwidth}{!}{
				\begin{tikzpicture}[every node/.style={scale = 1}, every 
				loop/.style={min 
					distance=1mm,in=60,out=90,looseness=10}] 
				\node (A) {A};
				\node (T) [right of=A] {T};
				\node (M) [below right = .5cm and .075cm of A] {M};
				\node (H) [below of=M] {H};
				\node (L) [left of=M] {L};
				\node (I) [above of=T] {I};
				\path
				(A) [->]edge node {} (T)
				(T) [->]edge [bend left=20] node {} (M)
				(L) [->]edge [bend left = 20] node {} (A)
				(A) [->]edge [bend left = 20] (L)
				(A) [->]edge node {} (M)
				(M) [->]edge node {} (L)
				(M) [->]edge node {} (H)
				(L) [->]edge [bend left = 10] node {} (H)
				(H) [->]edge [bend left = 20] node {} (L)
				(I) [->]edge node {} (T)
				(M) [->]edge [bend left=20] node {} (T)
				(A) [->] edge [loop above, looseness = 5] (A)
				(I) [->] edge [loop right, looseness = 5] (I)
				(T) [->] edge [loop right, looseness = 5] (T)
				(L) [->] edge [loop left, looseness = 5] (L)
				(M) [->] edge [loop right, looseness = 5] (M)
				(H) [->] edge [loop right, looseness = 5] (H)
				;
				\end{tikzpicture}}
			\caption{The directed graph of Example
                          \ref{exmp:gw} illustrating a model where
                          marijuana ($M$) potentially acts as a gateway
                          drug, while alcohol ($A$) as well as
                          tobacco ($\, T$) do not directly affect 
                          hard drug use.}
			\label{fig:gwDG}
		\end{figure}
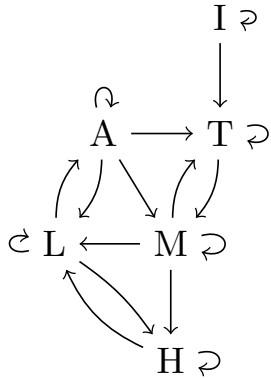
	
	We propose a hypothetical, dynamical model of transitions into
        abuse via a gateway drug, and 
	more 
	generally, a model of substance abuse progression. Substance abuse is 
	known to be 
	associated with social 
	factors, genetics, and other individual and environmental factors 
	\cite{who2004}. Substance abuse can evolve over time when an individual 
	starts or stops using some drug. In this example, we consider
	substance processes Alcohol ($A$), Tobacco ($T$), Marijuana ($M$), and Hard 
	drugs 
	($H$) modeled as zero-one processes, that is, stochastic processes 
	that are piecewise constantly equal to zero (no substance use) or one 
	(substance use). We also include $L$, a process describing 
	life events, and a process $I$, which can be thought 
	of as an exogenous process that influences the tobacco consumption of 
	the individual,
	e.g. the price of tobacco which may change due to changes in tobacco 
	taxation. Let $V = \{A, T, M, H, L, I\}$.
	
	We will visualize each process as a node in a graph and draw
        an arrow from one process to another if the first has a direct
        influence on the second. We will not go into a full discussion
        of how to formalize ``influence'' in terms of a continuous-time causal 
        dynamical model as this would lead us astray, see instead
        \cite{Didelez2008,didelez2015,sokol2014}. The upshot is
        that for a (faithful) causal model, there is no direct influence if and 
        only if $\alpha \not\rightarrow \beta \mid V\setminus
        \{\alpha\}$, which identifies the ``influence'' graph with the
        local independence graph.

	Several formalizations of the gateway drug 
	question are possible. We 
	will focus on the questions ``is the use of hard drugs locally independent 
	of use of alcohol for some conditioning set?'' and ``is the use of hard 
	drugs 
	locally independent of 
	the use of tobacco for some conditioning set?''. Using the dynamical nature 
	of local independence, we are asking if e.g. the past alcohol usage
	changes the hard drug usage propensity when accounting for the past of 
	all 
	other 
	processes in the model. This is one possible formalization of the gateway 
	drug question as a negative answer would mean that there exist some 
	gateway 
	processes through which any influence of alcohol usage on hard drug usage
	is mediated. If the visualization in Figure \ref{fig:gwDG} is indeed a 
	local independence graph in the above sense we see that 
	conditioning on all other processes, $H$ is indeed locally independent of 
	$A$ 
	and locally independent of $T$. In this hypothetical scenario we could 
	interpret this as marijuana in fact acting as a gateway drug to hard drugs. 
	If the global Markov property holds, we can furthermore use 
	$\mu$-separation to obtain further local independences from the graph.
	We return to this example in Section \ref{sec:dmeg} to illustrate how 
		the 
		main results of the paper can be applied. In particular, we are 
		interested in what conclusions we can make when we do not observe all 
		the processes but only a subset.
				
		\label{exmp:gw}
\end{exmp}

\subsection{Marginalization and separability}

\begin{defn}[Marginalization]
Given an independence model $\mathcal{I}$ over $V$, 
the marginal independence model over $O\subseteq V$ is 
defined as 
$$
\mathcal{I}^O = \{ \langle A, B\mid C \rangle \mid \langle A, B\mid C \rangle
\in \mathcal{I};\ A,B,C \subseteq O  \}.
$$
\end{defn}

Marginalization is defined abstractly above, though we are primarily
interested in the marginalization of the independence model
encoded by a local independence graph via $\mu$-separation. The main
objective is to obtain a graphical representation of such a
marginalized independence model involving only the nodes
$O$. To this end, we consider the notion of separability in an independence 
model. 

\begin{defn}[Separability]
	Let $\mathcal{I}$ be an independence model 
	over $V$. Let $\alpha, \beta \in V$. We say that $\beta$ 
	is {\it separable} from $\alpha$ 
	if there exists $C\subseteq V\setminus \{\alpha\}$ such that 
	$\ind{\alpha}{\beta}{C} \in \I$, and 
	otherwise we say that $\beta$ is {\it 
		inseparable} from $\alpha$. We define
	
	$$s(\beta, \mathcal{I}) = \{\gamma\in V\mid \text{$\beta$ is 
		separable from $\gamma$}\}.$$
	
	\noindent We also define $u(\beta, \mathcal{I}) = V\setminus 
	s(\beta, 
	\mathcal{I})$.
	\label{def:sep}
\end{defn}

We show in Proposition \ref{prop:musepDG} that if $\mathcal{I}$ is the
independence model induced by a directed graph via $\mu$-separation,
then $\alpha \in u(\beta, \mathcal{I})$ if and only if there is a directed edge 
from $\alpha$ to $\beta$. In this case the graph is thus directly
identifiable from separability properties of $\mathcal{I}$. That is,
however, not true in general for a marginalization of
$\mathcal{I}$, and this is the motivation for developing a theory of
directed \emph{mixed} graphs with $\mu$-separation.

\subsection{Graph theory}

A graph, $\mathcal{G} = (V,E)$, is an ordered pair where $V$ is a finite set of 
vertices (also called nodes) and $E$ is a finite set of edges. Furthermore, 
there is a map that to 
each edge assigns a pair of nodes (not necessarily 
distinct). We say that the edge is {\it between} these two nodes.
We consider graphs with two types of edges: directed ($\rightarrow$) and
bidirected ($\leftrightarrow$). We can think of the edge 
set as a disjoint 
union, $E = E_d \disjU E_b$, 
where $E_d$ is a set of ordered pairs of nodes $(\alpha, \beta)$ 
corresponding to directed edges, and $E_b$ is a set of 
unordered 
pairs of nodes $\{\alpha, \beta\}$ corresponding to bidirected
edges. This implies that the edge $\alpha \leftrightarrow \beta$ is 
identical to the edge $\beta \leftrightarrow \alpha$, but the edge $\alpha 
\rightarrow \beta$ is different from the edge $\beta\rightarrow \alpha$. It 
also implies that the graphs 
we consider can have multiple edges between a pair of nodes $\alpha$ and 
$\beta$, but they will always be a subset of the 
edges $\{\alpha \rightarrow \beta, \alpha \leftarrow \beta, \alpha 
\leftrightarrow \beta\}$.

\begin{defn}[DMG]
	A {\it directed mixed graph} (DMG), $\mathcal{G} = (V,E),$ is a graph with 
	node set $V$ and edge set $E$ consisting of directed and bidirected edges 
	as described above. 
	\label{def:DMG}
\end{defn}

Throughout the paper, $\G$ will denote a DMG with node set $V$
and edge set $E$. Occasionally, we will also use $\mathcal{D}$ and 
$\mathcal{M}$ to denote DMGs. We use $\mathcal{D}$ only when the DMG is also a 
directed graph, that is, has no bidirected edges. We use $\mathcal{M}$ to 
stress that some 
DMG is obtained as a marginalization of a DMG 
on 
a 
larger node set. We will use notation such as
$\leftrightarrow_{\mathcal{G}}$ or $\rightarrow_{\mathcal{D}}$ to denote the 
specific
graph that an edge belongs to.

If $\alpha \rightarrow \beta$, we say that the edge has a 
{\bf tail} at $\alpha$ and a {\bf head} at $\beta$. Jointly tails and heads are 
called (edge) {\bf marks}. An edge $e\in E$ between 
nodes $\alpha$ and $\beta$ is a {\bf loop} if $\alpha = \beta$. We also say 
that the edge is {\bf incident} with the 
node $\alpha$ and with the node $\beta$ and that $\alpha$ and $\beta$ are {\bf 
	adjacent}.

For $\alpha, \beta \in V$ we use the notation $\alpha \sim \beta$ to
denote a generic edge of any type between $\alpha$ and $\beta$. We use
the notation $\alpha\ *\!\!\rightarrow \beta$ to indicate an edge that
has a head at $\beta$ and may or may not have a head at
$\alpha$. Note that the presence of one edge, $\alpha
\rightarrow \beta$, say, does not in general
preclude the presence of other edges between these two nodes. Finally, $\alpha 
\notstarrightarrow_\mathcal{G} \beta$ means
that there is no edge in $\G$ between $\alpha$ and $\beta$ that has a
head at $\beta$ and $\alpha\not\rightarrow_\G\beta$ means that there is no 
directed edge from $\alpha$ to $\beta$. Note that 
$\alpha\not\rightarrow_\G\beta$ is a statement about the absence of an edge in 
the graph $\G$ and to avoid confusion with local independence, 
$\alpha\not\rightarrow\beta\mid C$, we 
always include the conditioning set when writing local independence statements, 
even if $C = \emptyset$ (see also Definition \ref{def:LIG}). 

We say that $\alpha$ is a {\bf parent} of $\beta$ in the graph
$\mathcal{G}$ if $\alpha \rightarrow \beta$ is present
in $\mathcal{G}$ and that $\beta$ is a {\bf child} of
$\alpha$. We say that $\alpha$ is a {\bf sibling} of $\beta$ (and that
$\beta$ is a sibling of $\alpha$) if $\alpha \leftrightarrow \beta$ is
present in the graph. The motivation of the term sibling will be
explained in Section \ref{sec:DMGs}. We use $\mathrm{pa}(\alpha)$ to denote the 
set of parents of $\alpha$.
	
	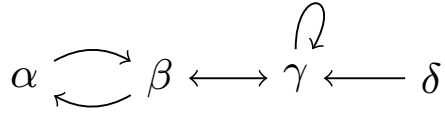
\begin{figure}
		\resizebox{0.5\textwidth}{!}{
			\begin{tikzpicture}[every node/.style={scale = 1}, every 
			loop/.style={min 
				distance=1mm,in=60,out=90,looseness=10}] 
			\node (a) {$\alpha$};
			\node (b) [right of=a] {$\beta$};
			\node (c) [right of=b] {$\gamma$};
			\node (d) [right of=c] {$\delta$};
			\path
			(a) [->] edge [bend left] (b)
			(a) [<-]edge [bend right] (b)
			(b) [<->] edge (c)
			(c) [->]edge [loop above] (c)
			(d) [->]edge (c);
			\end{tikzpicture}}
		\caption{A directed mixed graph with node set $\{\alpha, \beta, \gamma, 
			\delta\}$. Consider first the walk $\alpha \rightarrow \beta$. This 
			is different from the walk $\beta \leftarrow \alpha$ as walks are 
			ordered. Consider instead the two walks $\beta \leftrightarrow 
			\gamma 
			\leftarrow \gamma \leftarrow \delta$ and $\beta \leftrightarrow 
			\gamma \rightarrow \gamma \leftarrow \delta$. These two walks have 
			the same (ordered) sets of nodes and edges but are not equal 
			as the loop at $\gamma$ has different orientations between the two 
			walks. Furthermore, one can note that for the first of the two 
			walks, $\gamma$ is a collider in 
			the first instance, but not in the second. The walks $\alpha 
			\rightarrow \beta \rightarrow \alpha$ and $\alpha \rightarrow \beta 
			\leftarrow \alpha$ are both cycles, and the second is an example of 
			the fact that the same edge can occur twice in a cycle.}
		\label{fig:walkWithLoop}
	\end{figure}

	A {\bf walk} is an ordered, alternating sequence of vertices, $\gamma_i$, 
	and 
	edges, 
	$e_j$, denoted $\omega = \langle \gamma_1, e_1, \ldots, e_n, \gamma_{n+1} 
	\rangle$, such 
	that 
	each $e_i$ is between $\gamma_i$ and $\gamma_{i+1}$, along with an 
	orientation of each directed loop along the walk (if $e_i$ is
        a loop then we also know if $e_i$ 
	points in the direction of $\gamma_1$ or in the direction of 
	$\gamma_{n+1}$).  Without the orientation, for instance the walks $\alpha 
	\rightarrow \beta \rightarrow 
	\beta \rightarrow \gamma$ 
	and 
	$\alpha \rightarrow \beta \leftarrow \beta \rightarrow \gamma$ would be 
	indistinguishable. See Figure 
	\ref{fig:walkWithLoop} for examples.  We will often present the walk $\omega$ using the notation
	
	$$
	\gamma_1 \overset{e_1}{\sim} \gamma_2 \overset{e_2}{\sim} \ldots 
	\overset{e_n}{\sim} \gamma_{n+1},
	$$
	
	\noindent where the loop orientation is explicit. We will omit
        the edge superscripts when they are not needed. 
%       If 
%	we consider some directed edge which is not a loop, then there is only one 
%	orientation of the edge which is consistent with the order of the nodes 
%	along the walk, 
%	and 
%	thus the 
%	orientation of the directed edges along the walk is strictly speaking only 
%	needed for loops. 

	We say 
	that the walk $\omega$ {\it contains} 
	nodes $\gamma_i$ and edges $e_j$. The {\bf length} of the walk is $n$, the 
	number of 
	edges 
	that it contains. We define a {\bf trivial walk} to be a 
	walk with no edges, and therefore only a single node. Equivalently, a 
	trivial 
	walk can be defined as a walk of length zero.  A {\bf subwalk}
        of $\omega$ is 
	either 
	itself a 
	walk of the form $\langle \gamma_k, e_k, \ldots, e_{m-1}, \gamma_m\rangle$ 
	where $1\leq k < m \leq n+1$ or a trivial walk $\langle \gamma_k \rangle$, 
	$1\leq k \leq n+1$.  
	A 
	(nontrivial) walk is 
	uniquely 
	identified 
	by its edges, and the ordering and orientation of these edges, hence the 
	vertices can be omitted 
	when describing the walk. At times we will omit the edges to simplify 
	notation, 
	however, we will always have a specific, uniquely identified walk in mind 
	even 
	when the edges and/or their orientation is omitted. The first and last 
	nodes of 
	a 
	walk are called {\bf 
		endpoint nodes} (these could be equal) or just endpoints, and we say 
		that a walk is {\it 
		between} its 
		endpoints, or alternatively {\it from} its first node {\it to} its last 
		node. We call the walk $\omega^{-1} = \langle \gamma_{n+1}, e_n, \ldots, e_1, \gamma_1 \rangle$ 
		the {\bf inverse} walk of $\omega$. Note that the orientation of directed loops is also 
		reversed in the inverse walk such that they point towards $\gamma_1$ in 
		the inverse if 
		and only if they point towards $\gamma_1$ in the original walk. A 
		{\bf 
			path} is a walk on which no node is 
			repeated. 

Consider a walk $\omega$ and a subwalk thereof,
			$\langle \alpha, e_1, \gamma, e_2, \beta\rangle$, where 
			$\alpha,\beta,\gamma \in 
			V$ and $e_1, e_2\in E$. If $e_1$ and $e_2$ both have heads at 
			$\gamma$ 
			then $\gamma$ is a {\bf collider} on $\omega$. If this is not the 
			case, then 
			$\gamma$ is a noncollider. Note that an endnode of a walk is 
			neither a 
			collider, nor a noncollider. We stress that the property of being 
			a 
			collider/noncollider is relative to a
                        walk (see also Figure \ref{fig:walkWithLoop}).
			
			Let $\omega_1 = \langle \alpha, e_1^1,\gamma_1^1 \ldots, 
			\gamma_{n-1}^1, e_n^1, 
			\beta 
			\rangle$ and $\omega_2 = \langle \alpha, e_1^2,\gamma_1^2 \ldots, 
			\gamma_{m-1}^2, e_m^2, \beta 
			\rangle$ be two (nontrivial) walks. We say 
			that they are 
			{\bf 
				endpoint-identical} if $e_1^1$ and $e_1^2$ have the same mark 
				at $\alpha$ and 
				$e_n^1$ and $e_m^2$ have the same mark at $\beta$. Note that 
				this may depend on the orientation of directed edges in the two 
				walks. Assume that 
				some edge $e$ is 
				between $\alpha$ and $\beta$. We say that the (nontrivial) 
				walk 
				$\omega_1$ is 
				endpoint-identical to $e$ if it is endpoint-identical to 
				the walk 
				$\langle \alpha, e, \beta \rangle$. If $\alpha = \beta$ and $e$ 
				is directed this 
				should hold for just one of the possible orientations of $e$. 
				
				Let $\omega_1$ be a walk between $\alpha$ and $\gamma$, and 
				$\omega_2$ a walk 
				between $\gamma$ and $\beta$. The {\bf composition} of 
				$\omega_1$ with 
				$\omega_2$ is the walk that starts at $\alpha$, traverses every 
				node and edge 
				of $\omega_1$, and afterwards every node and edge of 
				$\omega_2$, ending in 
				$\beta$. We say that we {\bf compose} $\omega_1$ with 
				$\omega_2$.
				
				A {\bf directed path} from $\alpha$ to $\beta$ is a path 
				between $\alpha$ and 
				$\beta$ consisting of edges of type $\rightarrow$ only 
				(possibly of length 
				zero) such that they all point in the direction of $\beta$. A 
				{\bf cycle} is either a loop, or a (nontrivial) path from 
				$\alpha$ to 
				$\beta$ composed with $\beta \sim \alpha$. 
				This means that in a cycle of length 2, an edge can be 
				repeated. A 
				{\bf directed 
					cycle} is either a loop, $\alpha \rightarrow \alpha$, or a 
					(nontrivial) 
					directed path from $\alpha$ to $\beta$ 
					composed with $\beta \rightarrow \alpha$.
					For $\alpha \in 
					V$ we let $An(\alpha)$ denote the set of {\bf ancestors}, 
					that is,
					
					\[
					An(\alpha) = \{\gamma \in V \mid \text{there is a directed 
					path 
					from $\gamma$ to $\alpha$ } \}.
					\]
					
					\noindent This is generalized to non-singleton sets 
					$C\subseteq V$,
					
					\[
					An(C) = \cup_{\alpha\in C} An(\alpha).
					\]
					
					\noindent We stress that $C \subseteq An(C)$ as we allow 
					for trivial
					directed paths in the 
					definition of an ancestor. We use the notation 
					$An_\mathcal{G}(C)$ if we wish 
					to emphasize in which graph 
					the ancestry is read, but omit the subscript when no 
					ambiguity arises.
					
					Let $\mathcal{G} = (V,E)$ be a graph, and let $O \subseteq 
					V$. Define  the 
						{\bf subgraph} induced by $O$ to be the graph 
						$\mathcal{G}_O 
						= (O, E_O)$ where 
						$E_O\subseteq E$ 
						is the set of edges that are between nodes in $O$. If 
						$\G_1 = (V,E_1)$ and 
						$\G_2 = (V,E_2)$, we will write $\G_1 \subseteq \G_2$ 
						to denote $E_1 
						\subseteq E_2$ and say that $\G_2$ is a {\bf 
						supergraph} of $\G_1$.

						A {\bf directed graph} (DG), $\mathcal{D} = 
						(V,E)$, 
						is a graph with only 
						directed 
						edges. Note that this also allows directed loops. 
						Within a class of graphs, we 
						define the 
						{\bf complete} graph to be the graph 
						which is the supergraph of all graphs in the class
						when such a graph 
						exists. For the class of DGs on 
						node set $V$, the complete graph is the graph 
						with edge set $E = \{ (\alpha,\beta) \mid 
						\alpha,\beta\in V \}$.
						
									A {\bf 
										directed acyclic graph} (DAG) is a DG 
										with no loops
									and no directed cycles. An {\bf 
										acyclic directed mixed graph} (ADMG) is 
										a DMG with no loops and 
									no 
									directed cycles.

\section{Directed mixed graphs and separation}
\label{sec:DMGs}

In this section we introduce $\mu$-separation for DMGs which are then
shown to be closed under marginalization. In particular, we obtain a DMG 
representing the 
independence model arising from a local
independence graph via marginalization.  

The class of DMGs contains as a subclass the ADMGs that have no directed cycles
\cite{evans2014, richardson2003}.  ADMGs have been used to represent
marginalized DAG models, analogously to how we will use DMGs to
represent margina\-lized DGs. ADMGs come with the
$m$-separation criterion which can be extended to DMGs, but this 
criterion differs 
in important ways from the $\mu$-separation criterion introduced below. These 
differences also mean that our main result on Markov equivalence does not apply 
to e.g. DMGs with $m$-separation and thus our theory of Markov equivalence 
hinges on the fact that we are considering DMGs using the asymmetric notion of 
$\mu$-separation. 

\subsection{$\mu$-separation}

We define $\mu$-separation as a generalization of $\delta$-sepa\-ra\-tion 
introduced by 
Didelez 
\cite{didelez2000}, analogously to how $m$-separation is a gene\-ralization of 
$d$-separation, see e.g. \cite{richardson2002}. In Section 
\ref{sec:otherSeps} we make the connection to Didelez's $\delta$-separation 
exact 
and 
elaborate further on this in Section \ref{sec:examp}.

\begin{defn}[$\mu$-connecting walk]
	A nontrivial walk $$\langle \alpha,e_1,\gamma_1,\ldots, 
	\gamma_{n-1},e_n,\beta\rangle$$
	in $\G$ 
	is said to be $\mu$-connecting (or simply {\it open}) 
	from $\alpha$ to $\beta$ 
	given $C$ if $\alpha\notin C$, every collider is in $An(C)$, no 
	noncollider is in $C$, and $e_n$ has a head at $\beta$.
	\label{def:muConnSeq}
\end{defn}

When a walk is not $\mu$-connecting given $C$, we say that it is {\it 
closed} 
or {\it blocked} by $C$. One should note that if $\omega$ is a $\mu$-connecting 
walk from
$\alpha$ to $\beta$ given $C$, the inverse walk, $\omega^{-1}$, is not
in general $\mu$-connecting from $\beta$ to $\alpha$ given $C$. The requirement 
that a
$\mu$-connecting walk be nontrivial, that is, of strictly positive
length, leads to the possibility of a node being separated from itself
by some set $C$ when applying the following graph separation criterion to
the class of DMGs.

\begin{defn}[$\mu$-separation] Let $A,B,C \subseteq V$. We say that $B$ is 
  $\mu$-separated from $A$ given $C$ if there is no $\mu$-connecting walk 
  from any $\alpha \in A$ to any $\beta\in B$ given $C$ and write 
  $\musep{A}{B}{C}$, or write
  $\musepG{A}{B}{C}{\G}$ if we want to stress to what graph the separation 
  statement applies.
\label{def:muSep} 
\end{defn}

The above notion of separation is given in terms of walks of which there are 
infinitely many in any DMG with a nonempty edge set. However, we will see that 
it is sufficient to consider a finite subset of walks from $A$ to $B$ 
(Proposition \ref{prop:muConnPath}). 

Given a DMG, $\mathcal{G} = (V,E)$, we define an 
independence model over 
$V$ using 
$\mu$-separation,

$$
\mathcal{I}(\mathcal{G}) = \{\langle A,B \mid C \rangle \mid 
\musep{A}{B}{C} \}.
$$

\noindent Definition \ref{def:muConnSeq} implies 
$\musep{A}{B}{C}$ 
whenever $A \subseteq C$ and therefore $\I(\G) \neq \emptyset$.

Below we state two propositions that essentially both give equivalent ways of 
defining $\mu$-separation. The propositions are useful when proving 
results on $\mu$-separation models.

\begin{prop}
	Let $\alpha,\beta\in V$, $C\subseteq V$. If there is a $\mu$-connecting 
	walk from $\alpha$ 
	to $\beta$ given $C$, then there is a $\mu$-connecting walk from $\alpha$ 
	to $\beta$ 
	that furthermore satisfies that every collider is in $C$.
	\label{prop:muConnAllCollInC}
\end{prop}

\begin{defn} A {\it route} from $\alpha$ to $\beta$ is a walk
  from $\alpha$ to $\beta$ such that no node different from 
  $\beta$ occurs more than
  once, and $\beta$ occurs at most twice.
\end{defn}

A route is always
a path, a cycle, or a composition of a path and a cycle that share
no edge and only share the vertex $\beta$.

\begin{prop}
  Let $\alpha, \beta \in V, C\subseteq V$. If $\omega$ is a $\mu$-connecting 
  walk
  from $\alpha$ to $\beta$ given $C$, then there is a $\mu$-connecting
  route from $\alpha$ to $\beta$ given $C$ consisting of edges in $\omega$. 
	\label{prop:muConnPath}
\end{prop}

If there is a $\mu$-connecting walk from $A$ to $B$ given $C$, it does not in 
general follow 
that we can also find a $\mu$-connecting path or cycle from $A$ to $B$ given 
$C$. As an 
example of this, consider the following DMG on nodes $\{\alpha,\beta,\gamma\}$: 
$\alpha \leftarrow \beta \leftarrow \gamma$.  There is a $\mu$-connecting walk 
from $\alpha$ to $\beta$ given $\emptyset$, and a $\mu$-connecting route, but 
no $\mu$-connecting path from $\alpha$ to $\beta$ given 
$\emptyset$.

\subsection{Marginalization of DMGs}

Given a DG or a DMG, $\G$, we are interested 
in 
finding a graph that represents the marginal independence model over a node set 
$O\subseteq V$, i.e., finding a graph $\mathcal{M}$ such that

\begin{equation}
\mathcal{I}(\mathcal{M}) = (\mathcal{I}(\mathcal{G}))^O.
\label{eq:margEq}
\end{equation}

It is well-known that the class of DAGs with $d$-separation is not closed under 
marginalization, 
i.e. for a DAG, $\mathcal{D} = (V,E)$, and $O 
\subsetneq V$, it is not in general possible to find a DAG with node set $O$ 
that encodes the same independence model among the variables in $O$ as did the 
original graph. Richardson and Spirtes \cite{richardson2002} gave a concrete 
counterexample and in 
Exam\-ple \ref{exmp:nonClosedDG} we give a similar example to make the 
analogous point: DGs read with $\mu$-separation are not closed under 
marginalization. In this example, we use the following proposition 
which gives a simple characterization of separability in DGs.

\begin{prop}
	Consider a DG, $\mathcal{D} = (V,E)$, and let $\alpha, \beta \in V$. Then 
	$\beta$ is $\mu$-separable (see Definition \ref{def:sep}) 
	from 
	$\alpha$ in 
	$\mathcal{D}$ if and only 
	if $\alpha \not\rightarrow_\mathcal{D} \beta$.
	\label{prop:musepDG}
\end{prop}

\begin{exmp}
	\begin{figure}
		\resizebox{0.5\textwidth}{!}{
			\begin{tikzpicture}[every node/.style={scale = 1}, every 
			loop/.style={min 
				distance=1mm,in=60,out=90,looseness=10}] 
			\node (a) {$\alpha$};
			\node (b) [right of=a] {$\beta$};
			\node (c) [right of=b] {$\gamma$};
			\node (d) [right of=c] {$\delta$};
			\node (e) [above right = .25cm of b] {$\varepsilon$};
			\path
			(a) [->]edge node {} (b)
			(d) [->]edge [loop above] (d)
			(d) [->]edge (c)
			(e) [->]edge (b)
			(e) [->]edge (c);
			\end{tikzpicture}}
		\caption{The directed graph of Example \ref{exmp:nonClosedDG} which 
			exemplifies that DGs are not closed under marginalization.}
		\label{fig:DG}
	\end{figure}
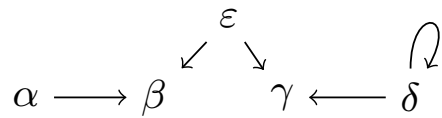
	
	Consider the directed graph, $\mathcal{G}$, in Figure \ref{fig:DG}. We wish 
	to show that it is not possible to encode the $\mu$-separations among nodes 
	in 
	$O=\{\alpha,\beta,\gamma,\delta\}$ using a DG on these nodes only. To 
	obtain a contradiction, assume 
	$\mathcal{D} = 
	(O,E)$ is a DG such that 
	
	\begin{align}
		\musepG{A}{B}{C}{\mathcal{D}} \Leftrightarrow 
		\musepG{A}{B}{C}{\G}
		\label{eq:margDGexample}
	\end{align}

	\noindent for $A,B,C \subseteq O$. There is no $C \subseteq O\setminus 
	\{\alpha\}$ such that $\musepG{\alpha}{\beta}{C}{\G}$ and no $C \subseteq 
	O\setminus \{\beta\}$ such that $\musepG{\beta}{\gamma}{C}{\G}$. If 
	$\mathcal{D}$ has the property (\ref{eq:margDGexample}) then it follows 
	from Proposition \ref{prop:musepDG} that $\alpha \rightarrow_\mathcal{D} 
	\beta$ and $\beta 
	\rightarrow_\mathcal{D} \gamma$. However, 
	then $\gamma$ is not $\mu$-separated from $\alpha$ given 
	$\emptyset$ in $\mathcal{D}$. This shows that there exists no DG, 
	$\mathcal{D}$, that satisfies (\ref{eq:margDGexample}).
	
	\label{exmp:nonClosedDG}
\end{exmp}

We note that marginalization of a
probability model does not only 
impose conditional independence constraints on the observed variables but also 
so-called equality and 
inequality constraints, see e.g. \cite{evans2016} and references therein. In 
this paper, we will only be concerned with the graphical representation of 
local 
independence constraints, and not with representing analogous equality or 
inequality 
constraints.

In the remainder of this section, we first introduce 
the {\it latent 
projection} of a graph, see also
\cite{vermaEquiAndSynthesis} and \cite{richardson2017}, and then show
that it provides a marginalized DMG in the sense of 
(\ref{eq:margEq}). At the end of the section, we give an algorithm 
for computing the latent projection of a DMG. 
This algorithm is an adapted version of one described by Sadeghi 
\cite{sadeghi2013} for 
a 
different class of graphs. Koster \cite{koster1999} described a similar 
algorithm for ADMGs.

\begin{defn}[Latent projection]
	Let $\mathcal{G} = (V,E)$ be a DMG, $V = M \disjU O$. We define the {\it 
	latent projection of $\G$ on $O$} to be the DMG $(O, D)$ such that $\alpha 
	\sim \beta \in D$ if and only 
	if 
	there exists an endpoint-identical (and nontrivial) walk between $\alpha$ 
	and $\beta$ in 
	$\G$ 
	with no colliders and such that every non-endpoint node is in $M$. Let 
	$m(\G, O)$ denote the latent projection of $\G$ on $O$.
\end{defn}

The definition of latent projection motivates the graphical term {\it sibling} 
for 
DMGs, as 
one way to obtain an edge $\alpha \leftrightarrow \beta$ is through 
a latent projection of 
a larger graph in which $\alpha$ and $\beta$ share a parent.

To characterize the class of graphs obtainable from a DG via a latent
projection, we introduce the {\it canonical DG of the DMG
  $\mathcal{G}$}, $\mathcal{C}(\mathcal{G})$, as follows: for each (unordered) 
  pair of nodes $\{\alpha, \beta\} \subseteq V$  such that
$\alpha \leftrightarrow_\mathcal{G} \beta$, add a distinct auxiliary node,
$m_{\{\alpha, \beta\}}$, add edges
$m_{\{\alpha, \beta\}} \rightarrow \alpha, m_{\{\alpha, \beta\}}\rightarrow 
\beta$ to $E$, and then remove all
bidirected edges from $E$. If 
$\mathcal{D}$ 
is any DG, then
$\mathcal{M} = m(\mathcal{D}, O)$ will satisfy

\begin{align}
	\alpha \leftrightarrow_\mathcal{M} \beta \Rightarrow \alpha 
	\leftrightarrow_\mathcal{M} 
	\alpha 
	\text{ for all } \alpha,\beta \in O
	\label{eq:mDGprop}
\end{align}

\noindent for all subsets of vertices $O$. Conversely, if $\G = (V, E)$ is a 
DMG that satisfies 
(\ref{eq:mDGprop}), then $\mathcal{G}$ is the latent projection of its
canonical DG; $m(\mathcal{C}(\mathcal{G}), V) = \mathcal{G}$. The class of DMGs 
that satisfy (\ref{eq:mDGprop}) is closed under marginalization (Proposition 
\ref{prop:closedMarg}) and has certain regularity properties (see e.g. 
Proposition \ref{prop:selfEdgeChar}). These result provide the 
means for graphically representing marginals of local independence graphs. However, the theory that leads to our 
main results on Markov equivalence does not require the property 
(\ref{eq:mDGprop}) and therefore we develop it for general DMGs.

\begin{prop}
	Let $O \subseteq V$.  The graph 
	$\mathcal{M} = m(\mathcal{G}, O)$ is a DMG. If $\G$ satisfies 
	(\ref{eq:mDGprop}), 
	then 
	$\mathcal{M}$ does as well.
	\label{prop:closedMarg}
\end{prop}

\begin{prop}
	Assume that $\mathcal{G}$ satisfies (\ref{eq:mDGprop}) and let $\alpha\in 
	V$. 
	Then 
	$\alpha$ has no loops if and only if $\musep{\alpha}{\alpha}{V\setminus 
	\{\alpha\}}$.
	\label{prop:selfEdgeChar}
\end{prop}

We also observe directly from the definition that the latent
projection operation preserves ancestry and non-ancestry in the
following sense.

\begin{prop}
	Let $O \subseteq V$, $\mathcal{M} = 
	m(\mathcal{G}, O)$ and $\alpha, \beta \in O$. Then $\alpha \in 
	An_\mathcal{G}(\beta)$ if and only if $\alpha \in An_\mathcal{M}(\beta)$.
	\label{prop:preserveAn}
\end{prop}

The main result of this section is the following theorem, which 
states that the marginalization defined by the latent 
projection operation preserves the marginal independence model encoded by a 
DMG. 

\begin{thm}
	Let $O \subseteq V$, $\mathcal{M} = 
	m(\mathcal{G}, 
	O)$. Assume $A,B,C 
	\subseteq O$. Then
	
	\[
	\musepG{A}{B}{C}{\G} \Leftrightarrow 
	\musepG{A}{B}{C}{\mathcal{M}}.
	\]
	\label{thm:DMGmarg}
\end{thm}

\subsection{A marginalization algorithm}

We describe an algorithm to compute the latent projection of a graph on some 
subset of nodes. For this purpose, we define a {\it triroute}, $\theta$, to be 
a walk of length 2, 
	$\langle \alpha, e_1, \gamma, e_2, \beta\rangle$, such that $\gamma 
	\neq \alpha, 
	\beta$. We suppress $e_1$ and 
	$e_2$ from the notation and use $_{\theta}(\alpha,\gamma, \beta)$ 
	to 
	denote the 
	triroute. We say that a triroute is {\it colliding} if 
	$\gamma$ is a 
	collider on $\theta$, and otherwise we say that it is {\it 
		noncolliding}. This is analogous to the concept of a {\it tripath} 
		(see e.g. \cite{lauritzen2017}), but allows for $\alpha = \beta$.

 Define 
$\Omega_M(\G)$ to be the set 
of noncolliding 
triroutes $_\theta(\alpha, m, \beta)$ such that $m\in M$ and such 
that an 
endpoint-identical edge $\alpha \sim \beta$ is not present in 
$\mathcal{G}$.

\begin{algorithm}
	\SetKwInOut{Input}{input}
	\SetKwInOut{Output}{output}
	\Input{a DMG, $\mathcal{G} = (V,E)$ \newline a subset $M\subseteq V$ over 
	which to marginalize}
	\Output{a graph $\mathcal{M} = (O, \bar{E}),\ O = V \setminus M$}
	Initialize $E_0 = E$, $\mathcal{M}_0 = (V,E_0)$, $k = 0$\;
	\While{$\Omega_M(\mathcal{M}_k) \neq \emptyset$}{
		Choose $\theta =\ _{\theta}(\alpha, m, \beta) \in 
		\Omega_M(\mathcal{M}_k)$\;
		Set $e_{k+1}$ to be the edge $\alpha \sim \beta$ which is 
		endpoint-identical to $\theta$\;
		Set $E_{k+1} = E_k \cup \{e_{k+1}\}$\;
		Set $\mathcal{M}_{k+1} = (V,E_{k+1})$\;
		Update $k = k + 1$	
	}
	\Return  $(\mathcal{M}_k)_{O}$
	\newline
	\label{algo:marg}
	\caption{Computing the latent projection of a DMG.}
\end{algorithm}

\begin{prop}
	Algorithm \ref{algo:marg} outputs the latent projection of a DMG.
	\label{prop:margAlgo}
\end{prop}

\section{Properties of DMGs}
\label{sec:propDMGs}

\begin{defn}[Markov equivalence]
	Let $\G_1=(V,E_1)$ and $\G_2=(V,E_2)$ be DMGs. We say 
	that $\G_1$ and $\G_2$ are {\it Markov equivalent} if 
	$\mathcal{I}(\mathcal{G}_1) = \mathcal{I}(\mathcal{G}_2)$. This defines an 
	equivalence relation and we let 
	$[\mathcal{G}_1]$ denote the (Markov) equivalence class of $\mathcal{G}_1$.
\end{defn}

\begin{exmp}[Markov equivalence in DGs]
	Let $\mathcal{D} = (V,E)$ be a DG. There is a directed edge from $\alpha$ 
	to 
	$\beta$ if and only if $\beta$ cannot be separated from $\alpha$ by any set 
	$C\subseteq V \setminus \{\alpha\}$ (Proposition \ref{prop:musepDG}). This 
	implies that two DGs are Markov 
	equivalent if and only if they are equal. Thus, in the restricted class of 
	DGs, every Markov equivalence class is a 
	singleton and in this sense {\it identifiable} from its induced 
	independence model. However, when 
	considering Markov equivalence in the more general class of DMGs not every 
	equivalence class of a DG is a singleton as the DG might be Markov 
	equivalent to a DMG. As an example of this, consider the complete DG 
	on a node set $V$ which is Markov 
	equivalent to the complete DMG on $V$.
\end{exmp}

\begin{defn}[Maximality of a DMG]
	We say that $\G$ is {\it maximal} if it is complete, or 
	if
	any added edge changes the induced independence model 
	$\I(\G)$.
	\label{def:maximalDMG}
\end{defn}

\subsection{Inducing paths}

Separability of nodes can be 
stu\-died using the concept of an {\it inducing path} which has also been 
used in
other classes of graphs 
\cite{ richardson2002, vermaEquiAndSynthesis}. In the context of DMGs and 
$\mu$-separation, it is 
natural to define several types of inducing paths due to the asymmetry of 
$\mu$-separation and the possibility of directed cycles in DMGs.

\begin{defn}[Inducing path]
	An {\it inducing path from $\alpha$ to $\beta$} is a 
	nontrivial 
	path or cycle, $\pi = \langle \alpha,\ldots,\beta\rangle$, which has a 
	head at $\beta$ and 
	such 
	that 
	there are no noncolliders on $\pi$ and every node is an 
	ancestor of $\alpha$ or $\beta$.
	The inducing path $\pi$ is {\it bidirected} if every edge on $\pi$ is 
	bidirected. If $\pi$ is not bidirected, it has one of the forms $\alpha 
	\rightarrow \beta$ or
	
	$$
	\alpha \rightarrow \gamma_1 \leftrightarrow \ldots \leftrightarrow \gamma_n 
	\leftrightarrow \beta.
	$$
	
	\noindent and we say that it is {\it unidirected}. If, furthermore,
	$\gamma_i 
	\in 
	An(\beta)$ 
	for all $i=1,\ldots,n$ (or it is on the form $\alpha \rightarrow \beta$) 
	then we say that it is {\it directed}.
\end{defn}

Note that an inducing path is by 
definition 
either a path or a cycle. An inducing path is either bidirected or unidirected. 
Some unidirected indu\-cing paths are also directed. Propositions 
\ref{prop:biIP} 
and \ref{prop:asymIP} show how bidirected and directed inducing paths in a 
certain sense correspond to bidirected and directed edges, respectively.

		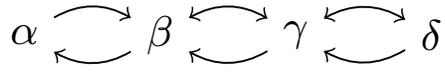
\begin{figure}
			\resizebox{0.5\textwidth}{!}{
				\begin{tikzpicture}[every node/.style={scale = 1}] 
				\node (a) {$\alpha$};
				\node (b) [right of=a] {$\beta$};
				\node (c) [right of=b] {$\gamma$};
				\node (d) [right of=c] {$\delta$};
				\path
				(b) [->]edge [bend left] node {} (a)
				(a) [->]edge [bend left] node {} (b)
				(b) [<->]edge [bend left] node {} (c)
				(c) [->]edge [bend left] node {} (b)
				(c) [<->]edge [bend left] node {} (d)
				(d) [->]edge [bend left] node {} (c);
				\end{tikzpicture}}
			\caption{Examples of inducing paths in a DMG: the path $\beta 
			\rightarrow \alpha$ is a unidirected inducing path from $\beta$ to 
			$\alpha$, and also a directed inducing path. The path $\beta 
			\leftrightarrow \gamma$ is a bidirected inducing path. The path 
			$\beta \leftrightarrow \gamma 
			\leftrightarrow \delta$ is a bidirected inducing path from $\beta$ 
			to $\delta$ (and by definition its inverse is a bidirected inducing 
			path from 
			$\delta$ to $\beta$). The 
			path $\delta \rightarrow \gamma \leftrightarrow \beta$ is both a 
			unidirected and a directed inducing path from $\delta$ to $\beta$, 
			whereas the path $\alpha 
			\rightarrow \beta \leftrightarrow \gamma$ is a unidirected inducing 
			path from $\alpha$ to $\gamma$, 
			but not a directed inducing path.}
			\label{fig:ipDMG}
		\end{figure}

\begin{prop}
	Let $\nu$ be an inducing path from $\alpha$ to $\beta$. The following holds 
	for any $C 
	\subseteq V \setminus \{\alpha\}$. If $\alpha \neq 
	\beta$, then there exists
	a $\mu$-connecting path from $\alpha$ to $\beta$ given $C$. If $\alpha = 
	\beta$ then there exists a $\mu$-connecting cycle 
	from $\alpha$ to $\beta$ 
	given $C$. We call such a path or cycle a {\it 
	$\nu$-induced open path or cycle}, respectively, or simply a
      $\nu$-induced open walk to cover both the case 
	$\alpha=\beta$ and the case $\alpha\neq\beta$. If the 
	inducing path is bidirected or directed, then the 
	$\nu$-induced open walk is endpoint-identical to the inducing path.
	\label{prop:IPepiConnPath}
\end{prop}

The following corollary is a direct consequence of Proposition
\ref{prop:IPepiConnPath}, show\-ing that $\beta$ is inseparable from
$\alpha$ if there is an inducing path from $\alpha$ to $\beta$ irrespectively
of whether the nodes are adjacent. 

\begin{cor}
	Let $\alpha,\beta\in V$. If there 
	exists an inducing path from $\alpha$ to $\beta$ in $\G$, then 
	$\beta$ is not $\mu$-separated from $\alpha$ given $C$ for any $C\subseteq 
	V\setminus \{\alpha\}$, that is, $\alpha \in u(\beta,
        \mathcal{I}(\mathcal{G}))$.  
	\label{cor:ipNonSep}
\end{cor}

The following two propositions  show that for two of the three types
of inducing paths there is a Markov equivalent supergraph in which the
nodes are adjacent. This illustrates how one can easily find
Markov equivalent DMGs that do not have the same adjacencies.
Example \ref{exmp:nonMax} shows that for a unidirected inducing path
it may not be possible to add an edge without changing the
independence model.

\begin{prop}
	If there exists a bidirected inducing path from $\alpha$ to $\beta$ in 
	$\G$, then adding $\alpha \leftrightarrow \beta$ in $\mathcal{G}$ does not 
	change the independence model.
	\label{prop:biIP}
\end{prop}

\begin{prop}
	If there exists a directed inducing path from 
	$\alpha$ to $\beta$ in $\G$, then adding $\alpha \rightarrow 
	\beta$ in 
	$\G$ does not change the independence model.
	\label{prop:asymIP}
\end{prop}

 We say 
 that nodes $\alpha$ and $\beta$ are {\it 
 	collider-connected} if there exists a nontrivial walk between
 $\alpha$ and $\beta$ such that every non-endpoint node is a 
 collider on the
 walk. We say that $\alpha$ is {\it directedly collider-connected} to
 $\beta$ if $\alpha$ and $\beta$ are collider-connected by a
 walk with a head at $\beta$.

\begin{defn}
	Let $\alpha,\beta \in V$. We define the set 
	$$
	D(\alpha,\beta) = \{\gamma\in An(\alpha,\beta)\mid \gamma 
	\text{ is directedly
	collider-connected to}\ \beta \} \setminus \{\alpha\}.
	$$
\end{defn}

Note that if $\alpha \not \rightarrow_\G \beta$, then $\mathrm{pa}(\beta)
\subseteq D(\alpha,\beta)$, and if the graph is furthermore a directed graph 
then $\mathrm{pa}(\beta)
= D(\alpha,\beta)$.

\begin{prop}
	If there is no inducing path from $\alpha$ to $\beta$ in $\G$, then 
	$\beta$ 
	is separated from $\alpha$ by $D(\alpha,\beta)$.
	\label{prop:noIPsep}
\end{prop}

\begin{exmp}[Inducing paths]
Consider the DMG 
on nodes $\{\alpha,\gamma\}$ and with a single edge $\gamma 
\rightarrow \alpha$. In this case, there is no inducing path from $\alpha$ to 
$\alpha$ and $\alpha$ is $\mu$-separated from $\alpha$ by $D(\alpha, \alpha) = 
\{\gamma\}$. Now add the edge $\alpha \leftrightarrow \gamma$. In this new
DMG, there is an inducing path from $\alpha$ to $\alpha$ and therefore $\alpha$ 
is inseparable from itself.
\end{exmp}

\begin{exmp}[Non-adjacency of inseparable nodes in a maximal DMG]

		\begin{figure}
		\resizebox{0.5\textwidth}{!}{
			\begin{tikzpicture}[every node/.style={scale = 1}] 
			\node (a) {$\alpha$};
			\node (b) [right of=a] {$\beta$};
			\node (c) [right of=b] {$\gamma$};
			\node (d) [right of=c] {$\delta$};
			\path
			(a) [->]edge node {} (b)
			(b) [->]edge [bend left] node {} (c)
			(c) [->]edge [bend left] node {} (b)
			(c) [<->]edge [bend left] node {} (d)
			(c) [->]edge [loop above, min 
			distance=1mm,in=60,out=90,looseness=10] node {} (c)
			(d) [->]edge [loop above, min 
			distance=1mm,in=60,out=90,looseness=10] node {} (d)
			(c) [<->]edge [loop below, min 
			distance=1mm,in=180+60,out=180+90,looseness=10] node {} (c)
			(d) [<->]edge [loop below, min 
			distance=1mm,in=180+60,out=180+90,looseness=10] node {} (d)
			(d) [->]edge [bend left] node {} (c);
			\end{tikzpicture}}
		\caption{A maximal DMG in which $\delta$ is 
		inseparable from $\beta$, 
		though no edge is between the two. See Example \ref{exmp:nonMax}. We 
		will in general omit the bidirected loops from the visual presentations 
		of DMGs, see also the discussion in Subsection \ref{ssec:dmeg}.}
		\label{fig:maxDMG}
		\end{figure}
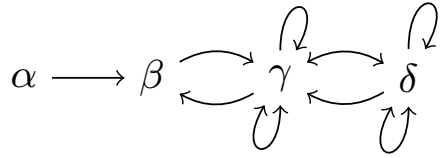

	Consider the DMG in Figure \ref{fig:maxDMG}. One can show that this DMG is 
	maximal (Definition \ref{def:maximalDMG}).
	There is an inducing path from $\beta$ to $\delta$ making $\delta$ 
	inseparable from $\beta$, yet no arrow can be 
	added 
	between $\beta$ and 
	$\delta$ without changing the independence model. This example illustrates 
	that maximal 
	DMGs 
	do not have the property that inseparable nodes are adjacent. This is 
	contrary to MAGs which form a subclass of ancestral graphs and have this
	exact property \cite{richardson2002}.
	\label{exmp:nonMax}
\end{exmp}

\section{Markov equivalence of DMGs}
\label{sec:MEofDMGs}

The main result of this section is that each Markov equivalence class of DMGs 
has a greatest element, that is, an element which is a supergraph of all other 
elements. This fact is helpful for understanding and graphically representing 
such equivalence classes, and potentially also for 
constructing learning algorithms. We will prove this result by arguing that the 
independence model of a DMG, $\G = (V,E)$, defines for each node 
$\alpha\in V$ a set of {\it potential parents} and a set of {\it potential 
siblings}. We then construct the greatest element of $[\G]$ by simply 
using these sets, and argue that this is in fact a Markov equivalent 
supergraph. As we only use the independence model to define the sets of 
potential parents and siblings, the supergraph is identical for all members of 
$[\G]$, and thus a greatest element. Within the equivalence class, the greatest 
element is also the only maximal element, and we will refer to it as 
the maximal element of the equivalence class.

\subsection{Potential siblings}

\begin{defn}
	Let $\I$ be an independence model over $V$ and let $\alpha, \beta \in V$. 
	We say that $\alpha$ and $\beta$ are {\it 
		potential siblings} in $\I$ if (s1)--(s3) hold:
	
	\begin{enumerate}[label=(s\arabic*)]
		\item $\beta \in u(\alpha, \mathcal{I})$ and 
		$\alpha \in u(\beta, \mathcal{I})$,
		\item for all $\gamma \in V$, $C \subseteq V$ such that $\beta \in C$,
		$$
		\ind{\gamma}{\alpha}{C} \in \I \Rightarrow \ind{\gamma}{\beta}{C}  \in 
		\I,
		$$
		\item for all $\gamma \in V$, $C \subseteq V$ such that $\alpha \in C$,
		$$
		\ind{\gamma}{\beta}{C} \in \I \Rightarrow \ind{\gamma}{\alpha}{C}  \in 
		\I.
		$$
	\end{enumerate}
	\label{def:potSib}
\end{defn}

Potential siblings are defined abstractly above in terms of the
independence model only. The following proposition gives a useful
characterization for graphical independence models by simply contraposing 
(s2) and (s3). 

\begin{prop}
	Let $\I(\G)$ be the independence model induced by $\G$. Then
	$\alpha, \beta \in V$ are potential siblings if and only if (gs1)--(gs3) 
	hold: 
	
	\begin{enumerate}[label=(gs\arabic*)]
		\item $\beta \in u(\alpha, \mathcal{I}(\G))$ and 
		$\alpha \in u(\beta, \mathcal{I}(\G))$,
		\item for all $\gamma \in V$, $C \subseteq V$ such that $\beta \in C$: 
		if there exists a $\mu$-connecting walk from $\gamma$ to $\beta$ given 
		$C$, then there exists a $\mu$-connecting walk from $\gamma$ to 
		$\alpha$ given $C$,
		\item for all $\gamma \in V$, $C \subseteq V$ such that $\alpha \in C$: 
		if there exists a $\mu$-connecting walk from $\gamma$ to $\alpha$ given 
		$C$, then there exists a $\mu$-connecting walk from $\gamma$ to 
		$\beta$ given $C$.
	\end{enumerate}
	\label{def:graphicalPotSib}
\end{prop}

\begin{prop}
	Assume that $\alpha \leftrightarrow \beta$ is in $\mathcal{G}$. Then 
	$\alpha$ and 
	$\beta$ are potential 
	siblings in $\I(\G)$. 
	\label{prop:SiIsPoSi}
\end{prop}

\begin{lem}
	Assume that $\alpha$ and $\beta$ are potential siblings in 
	$\mathcal{I}(\mathcal{G})$. Let $\mathcal{G}^+$ denote the DMG obtained 
	from $\mathcal{G}$ by adding $\alpha \leftrightarrow
	\beta$. Then 
	$\mathcal{I}(\mathcal{G}) = \mathcal{I}(\mathcal{G}^+)$.
	\label{lem:AddPoSiME}
\end{lem}

The above shows that if $\alpha$ and $\beta$ are potential siblings in $\I(\G)$ 
then there exists a supergraph, $\G^+$, which is Markov equivalent with $\G$ such 
that $\alpha$ and $\beta$ are siblings in $\G^+$. This motivates the term 
{\it potential} siblings.

\subsection{Potential parents}

In this section, we will argue that also a set of {\it potential parents} are 
determined by the independence model. This case is slightly more involved for 
two reasons. First, the relation 
is asymmetric, as for each potential parent edge there is a parent node and a 
child node. Second, adding directed edges potentially changes the ancestry of 
the graph.

\begin{defn}
	Let $\I$ be an independence model over $V$ and let $\alpha, \beta \in V$. 
	We say that $\alpha$ is a {\it potential parent} of $\beta$ in $\I$ if 
	(p1)--(p4) hold:

	\begin{enumerate}[label=(p\arabic*)]
		\item $\alpha \in u(\beta, \mathcal{I})$,
		\item for all $\gamma \in V$, $C \subseteq V$ such that $\alpha\notin 
		C$,
		$$
		\ind{\gamma}{\beta}{C} \in \I \Rightarrow \ind{\gamma}{\alpha}{C} \in 
		\I,
		$$
		\item for all $\gamma,\delta \in V$, $C \subseteq V$ such that 
		$\alpha \notin C, \beta \in C$,
		\begin{align*}
		\ind{\gamma}{\delta}{C} \in \I  \Rightarrow
		\ind{\gamma}{\beta}{C} \in \I \lor 
		\ind{\alpha}{\delta}{C} \in \I ,
		\end{align*}
		\item for all $\gamma \in V, C\subseteq V$, such that 
		$\alpha \notin C$,
		$$
		\ind{\beta}{\gamma}{C} \in \I \Rightarrow 
		\ind{\beta}{\gamma}{C \cup \{\alpha\}} \in \I.
		$$
	\end{enumerate}
	\label{def:potPa}
\end{defn}

\begin{prop}
	Let $\I(\G)$ be the independence model induced by $\G$. Then 
	$\alpha \in V$ is a potential parent of $\beta \in V$ if and only if 
	(gp1)--(gp4) 
	hold: 
	
	\begin{enumerate}[label=(gp\arabic*)]
		\item	$\alpha \in u(\beta, \mathcal{I}(\G))$,
		\item for all $\gamma \in V$, $C \subseteq V$ such that $\alpha\notin 
		C$: if there exists a $\mu$-connecting walk from $\gamma$ to $\alpha$ 
		given $C$, then there exists a $\mu$-connecting walk from $\gamma$ to 
		$\beta$ given $C$,
		\item for all $\gamma,\delta \in V$, $C \subseteq V$ such that 
		$\alpha \notin C, \beta \in C$: if there exists a $\mu$-connecting walk 
		from $\gamma$ to $\beta$ given $C$ and a $\mu$-connecting walk from 
		$\alpha$ to $\delta$ given $C$, then there exists a $\mu$-connecting 
		walk from $\gamma$ to $\delta$ given $C$,
		\item for all $\gamma \in V, C\subseteq V$, such that 
		$\alpha \notin C$: if there exists a $\mu$-connecting walk from $\beta$ 
		to $\gamma$ given $C \cup \{\alpha\}$, then there exists a 
		$\mu$-connecting walk from $\beta$ to $\gamma$ given $C$.
	\end{enumerate}
	\label{def:graphicalPotPa}
\end{prop}

\begin{prop}
	Assume that $\alpha \rightarrow \beta$ is in $\mathcal{G}$. Then $\alpha$ 
	is a 
	potential parent of $\beta$ in $\mathcal{I}(\mathcal{G})$. 
	\label{prop:PaIsPoPa}
\end{prop}

\begin{lem}
	Assume that $\alpha$ is a potential parent of $\beta$ in 
	$\mathcal{I}(\mathcal{G})$. Let $\mathcal{G}^+$ denote the DMG obtained 
	from $\mathcal{G}$ by adding $\alpha \rightarrow \beta$. Then 
	$\mathcal{I}(\mathcal{G}) = \mathcal{I}(\mathcal{G}^+)$.
	\label{lem:AddPoPaME}
\end{lem}

\subsection{A Markov equivalent supergraph}

Let $\mathcal{G} = (V,E)$ be a DMG. Define 
$\mathcal{N}(\mathcal{I}(\mathcal{G})) = (V,E^m)$ to 
be the DMG with edge set $E^m = E^d \cup E^b$ where $E^d$ is a set of 
directed edges and $E^b$ a set of bidirected edges such that the directed edge 
from $\alpha$ to $\beta$ is in $E^d$ if and only if $\alpha$ is a potential 
parent of $\beta$ in $\mathcal{I}(\mathcal{G})$ and the bidirected edge between 
$\alpha$ 
and $\beta$ is in $E^b$ if and only if $\alpha$ and $\beta$ are potential 
siblings in $\mathcal{I}(\mathcal{G})$.

\begin{thm}
	Let $\mathcal{N} = 
	\mathcal{N}(\mathcal{I}(\mathcal{G}))$. Then $\mathcal{N}\in 
	[\mathcal{G}]$ and $\mathcal{N}$ is a supergraph of all elements of 
	$[\mathcal{G}]$. Furthermore, if we have a finite sequence of DMGs
	$\mathcal{G}_0, \mathcal{G}_1, \ldots, \mathcal{G}_m$, $\mathcal{G}_i = 
	(V,E_i)$, such that $\mathcal{G}_0 = \mathcal{G}$, $\mathcal{G}_m = 
	\mathcal{N}$, and $E_i \subseteq E_{i+1}$ for all $i =0, \ldots, m -1$, 
	then $\G_i$ is Markov equivalent with $\mathcal{N}$ for all $i =0, \ldots, 
	m -1$.
	\label{thm:maxElm}
\end{thm}

The graph $\mathcal{N}$ in the above theorem is a supergraph of every Markov 
equivalent DMG and therefore maximal. On the other hand, every maximal DMG is 
a representative of its equivalence class, and also a supergraph of all Markov 
equivalent DMGs. This means that we can use the class of maximal DMGs to obtain 
a unique representative for each DMG equivalence 
class.

Lemmas \ref{lem:AddPoSiME} and \ref{lem:AddPoPaME} show that conditions 
(gs1)--(gs3) and (gp1)--(gp4) are sufficient to Markov equivalently add a 
bidirected or a directed 
edge, respectively. The conditions are also necessary in the sense that for 
each 
condition one can find example graphs where only a single condition is violated 
and 
where the larger graph is not Markov equivalent to the smaller graph.

We can note that $\alpha$ is a potential parent and a potential sibling of 
$\alpha$ if and only if $\alpha \in u(\alpha, \I(\G))$. This means that in 
$\mathcal{N}(\mathcal{I}(\mathcal{G}))$ for each node either both a directed 
and a bidirected loop is present or no loop at all.

\subsection{Directed mixed equivalence graphs}
\label{ssec:dmeg}

Theorem \ref{thm:maxElm} suggests that one can represent an equivalence class 
of DMGs by displaying the maximal element and then 
simply indicate which edges are not present for all members of the 
equivalence class.

\begin{defn}[DMEG]
	Let 
	$\mathcal{N} = (V,F)$ be a maximal DMG. 
	Define 
	$\bar{F}\subseteq F$ such that for $e\in F$ we let $e\in 
	\bar{F}$ if and only if there exists a DMG $\mathcal{G} = 
	(V,\tilde{F})$ such that $\G \in [\mathcal{N}]$ and $e\notin \tilde{F}$. 
	We call 
	$\mathcal{N}' = (V,F,\bar{F})$ a {\it directed mixed equivalence graph} 
	(DMEG). When visualizing $\mathcal{N}'$, we draw $\mathcal{N}$, but use 
	dashed edges 
	for the set $\bar{F}$, see Figure \ref{fig:DMEG}.
	\label{def:markDMG}
\end{defn}

Let $\mathcal{N}' = 
(V,F,\bar{F})$ be a DMEG.  The 
DMG $(V,F)$ is in the equivalence class represented by $\mathcal{N}'$. However, 
one cannot necessarily remove any subset of $\bar{F}$ and obtain a member of 
the Markov equivalence class (see Fi\-gu\-re \ref{fig:DMEG}). Moreover, an 
equivalence class does not in general contain a 
least 
element, that is, an element which is a subgraph of all Markov equivalent 
graphs.

We will throughout this section let $\mathcal{N} = (V,F)$ be a maximal DMG. For 
$e\in F$ we will use $\mathcal{N} - e$ to denote the graph $(V, F\setminus 
\{e\})$. Assume that we have a maximal DMG
from which we wish to derive the DMEG. Consider some 
edge 
$e\in F$. If $\mathcal{N} - e \in [\mathcal{N}]$, then $e\in 
\bar{F}$ as there exists a Markov 
equivalent subgraph of $\mathcal{N}$ in which $e$ is not present. On the other 
hand, if $\mathcal{N} - e \notin [\mathcal{N}]$ then we note that 
$\mathcal{N} - e$ is the largest subgraph of $\mathcal{N}$ that does not 
contain 
$e$. Let $\mathcal{K}$ be a subgraph of $\mathcal{N}$ that does not contain 
$e$. 
Then $\mathcal{I}(\mathcal{N}) \subsetneq \mathcal{I}(\mathcal{N} - e) 
\subseteq 
\mathcal{I}(\mathcal{K})$. Using Theorem \ref{thm:maxElm}, we know that all 
$\mathcal{N}$-Markov equivalent DMGs are in fact subgraphs of $\mathcal{N}$, 
and using that $\mathcal{K}$ is not Markov equivalent to $\mathcal{N}$ we see 
that all graphs in $[\mathcal{N}]$ must contain $e$. This 
means that when $\mathcal{N} - e \notin [\mathcal{N}]$ then $e\notin \bar{F}$ 
as 
$e$ must be present in all Markov equivalent DMGs.

Any loop should in principle be dashed when drawing a DMEG as for each 
node in 
a maximal DMG either both the 
directed and the bidirected loop is present or neither of 
them. However, we choose to not present them as dashed as if they are present 
in the maximal DMG, then at least one of them 
will be present in any Markov equivalent DMG satisfying (\ref{eq:mDGprop}), 
that 
is, for any DMG which is a marginalization of a DG. In addition we only draw 
the 
directed 
loop 
to not overload the visualizations.

\begin{center}
	\begin{figure}[t]
\resizebox{\textwidth}{!}{%
\begin{tikzpicture}[scale = .6, every node/.style={circle, inner sep = .2mm}, 
font=\larger, bend angle = 15]

\def\x{\textwidth/3}
\def\y{3.5cm}
\def\xx{.1cm}
\def\yy{-.5cm}

% picture 1
\node (a) {$\alpha$};
\node (b) [right of=a] {$\beta$};
\node (c) [below of=a] {$\gamma$};
\node (d) [below of=b] {$\delta$};
\node (x) [above left of=a, xshift = \xx, yshift = \yy, circle, draw, inner 
sep = 1pt, scale=0.7] {3};
\path
(a) [->]edge [bend left] (b)
(a) [<->]edge [bend right] (b)
(b) [<-]edge node {} (c)
(c) [->]edge node {} (d)
(d) [->]edge [loop right] node {} (d)
(a) [->]edge [loop left] node {} (a)
(b) [->]edge [loop right] node {} (b)
%(d) [->]edge [bend left] node {} (b)
(d) [<->]edge [bend right] node {} (b);

% picture 2
\begin{scope}[xshift=\x]
	\node (a) {$\alpha$};
	\node (b) [right of=a] {$\beta$};
	\node (c) [below of=a] {$\gamma$};
	\node (d) [below of=b] {$\delta$};
	\node (x) [above left of=a, xshift = \xx, yshift = \yy, circle, draw, 
	inner 
	sep = 1pt, scale=0.7] {6};
	\path
	%(a) [->]edge [bend left] node {} (b)
	(a) [<->]edge [bend right] node {} (b)
	(b) [<-]edge node {} (c)
	(c) [->]edge node {} (d)
	(d) [->]edge [loop right] node {} (d)
	(a) [->]edge [loop left] node {} (a)
	(b) [->]edge [loop right] node {} (b)
	%(d) [->]edge [bend left] node {} (b)
	(d) [<->]edge [bend right] node {} (b);
\end{scope}

% picture 3
\begin{scope}[yshift=\y]
	\node (a) {$\alpha$};
	\node (b) [right of=a] {$\beta$};
	\node (c) [below of=a] {$\gamma$};
	\node (d) [below of=b] {$\delta$};
	\node (x) [above left of=a, xshift = \xx, yshift = \yy, circle, draw, 
	inner 
	sep = 1pt, scale=0.7] {2};
	\path
	(a) [->]edge [bend left] node {} (b)
	(a) [<->]edge [bend right] node {} (b)
	%(b) [<-]edge node {} (c)
	(c) [->]edge node {} (d)
	(d) [->]edge [loop right] node {} (d)
	(a) [->]edge [loop left] node {} (a)
	(b) [->]edge [loop right] node {} (b)
	(d) [->]edge [bend left] node {} (b)
	(d) [<->]edge [bend right] node {} (b);
\end{scope}

% picture 4
\begin{scope}[xshift=\x, yshift=\y]
	\node (a) {$\alpha$};
	\node (b) [right of=a] {$\beta$};
	\node (c) [below of=a] {$\gamma$};
	\node (d) [below of=b] {$\delta$};
	\node (x) [above left of=a, xshift = \xx, yshift = \yy, circle, draw, 
	inner 
	sep = 1pt, scale=0.7] {5};
	\path
	%(a) [->]edge [bend left] node {} (b)
	(a) [<->]edge [bend right] node {} (b)
	%(b) [<-]edge node {} (c)
	(c) [->]edge node {} (d)
	(d) [->]edge [loop right] node {} (d)
	(a) [->]edge [loop left] node {} (a)
	(b) [->]edge [loop right] node {} (b)
	(d) [->]edge [bend left] node {} (b)
	(d) [<->]edge [bend right] node {} (b);
\end{scope}

% picture 5
\begin{scope}[yshift=2*\y]
	\node (a) {$\alpha$};
	\node (b) [right of=a] {$\beta$};
	\node (c) [below of=a] {$\gamma$};
	\node (d) [below of=b] {$\delta$};
	\node (x) [above left of=a, xshift = \xx, yshift = \yy, circle, draw, 
	inner 
	sep = 1pt, scale=0.7] {1};
	\path
	(a) [->]edge [bend left] node {} (b)
	(a) [<->]edge [bend right] node {} (b)
	(b) [<-]edge node {} (c)
	(c) [->]edge node {} (d)
	(d) [->]edge [loop right] node {} (d)
	(a) [->]edge [loop left] node {} (a)
	(b) [->]edge [loop right] node {} (b)
	(d) [->]edge [bend left] node {} (b)
	(d) [<->]edge [bend right] node {} (b);
\end{scope}

% picture 6
\begin{scope}[xshift=\x, yshift=2*\y]
	\node (a) {$\alpha$};
	\node (b) [right of=a] {$\beta$};
	\node (c) [below of=a] {$\gamma$};
	\node (d) [below of=b] {$\delta$};
	\node (x) [above left of=a, xshift = \xx, yshift = \yy, circle, draw, 
	inner 
	sep = 1pt, scale=0.7] {4};
	\path
	%(a) [->]edge [bend left] node {} (b)
	(a) [<->]edge [bend right] node {} (b)
	(b) [<-]edge node {} (c)
	(c) [->]edge node {} (d)
	(d) [->]edge [loop right] node {} (d)
	(a) [->]edge [loop left] node {} (a)
	(b) [->]edge [loop right] node {} (b)
	(d) [->]edge [bend left] node {} (b)
	(d) [<->]edge [bend right] node {} (b);
\end{scope}

% picture 7
\begin{scope}[xshift=2*\x, yshift=\y]
	\node (a) {$\alpha$};
	\node (b) [right of=a] {$\beta$};
	\node (c) [below of=a] {$\gamma$};
	\node (d) [below of=b] {$\delta$};
	\node (x) [above left of=a, xshift = \xx, yshift = \yy, circle, draw, 
	inner 
	sep = 1pt, scale=0.7] {7};
	\path
	(a) [->]edge [bend left, dashed] node {} (b)
	(a) [<->]edge [bend right] node {} (b)
	(b) [<-]edge [dashed] (c)
	(c) [->]edge node {} (d)
	(d) [->]edge [loop right] node {} (d)
	(a) [->]edge [loop left] node {} (a)
	(b) [->]edge [loop right] node {} (b)
	(d) [->]edge [bend left, dashed] node {} (b)
	(d) [<->]edge [bend right] node {} (b);
\end{scope}
\end{tikzpicture}
}
\caption{The DMG \protect\circled{\emph{1}} is maximal (the bidirected loops at 
$\alpha$, $\beta$ and $\delta$ have been omitted from the visual presentation). 
The 
DMGs
  \protect\circled{\emph{1}}\!\!\ --\!\ \protect\circled{\emph{6}} are the 
  six 
  elements 
  of its Markov 
equi\-va\-lence class (when ignoring Markov equivalent removal of loops). The 
graph 
\protect\circled{\emph{7}} is the corresponding 
DMEG. In a DMEG, every 
solid edge is in every graph in the equivalence class, every absent 
edge is not in any graph, and every dashed edge is in some, but not in 
others. Note that every DMG in the above equivalence class contains the edge 
$\gamma 
\rightarrow \beta$ or the edge $\delta \rightarrow \beta$ even though both are 
dashed in the DMEG. This example shows that not every equivalence class 
contains 
a 
least element.}
\label{fig:DMEG}
	\end{figure}
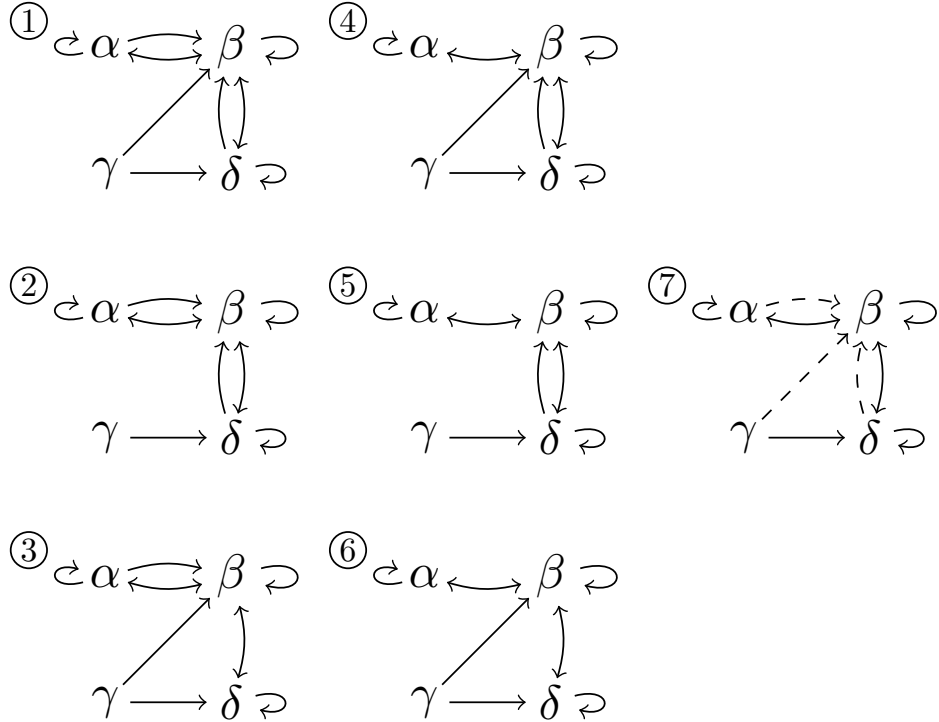
\end{center}

\subsection{Constructing a directed mixed equivalence graph}
\label{sec:dmeg}

When constructing a DMEG from $\mathcal{N}$ it suffices 
to consider 
the graphs $\mathcal{N} - e$ for each $e\in E$ and 
determine if 
they are Markov equivalent to $\mathcal{N}$ or not. A brute-force approach 
to 
doing so is to simply check all separation statements in both graphs. 
However, one can make a considerably more 
efficient algorithm.

\begin{prop}
	Assume $\alpha 
	\overset{e}{\rightarrow}_\mathcal{N} \beta$. It holds that $\mathcal{N} - e 
	\in 
	[\mathcal{N}]$ if and only if $\alpha \in 
	u(\beta, \mathcal{I}(\mathcal{N} - e))$.
	\label{prop:dirEdgeSepInN}
\end{prop}

\begin{prop}
	Assume $\alpha \overset{e}{\leftrightarrow}_\mathcal{N} \beta$. Then 
	$\mathcal{N} - e 
	\in 
	[\mathcal{N}]$ if and only if $\alpha \in u(\beta, \I(\mathcal{N} - e))$ 
	and $\beta \in u(\alpha, \I(\mathcal{N} - e))$.
	\label{prop:bidirEdgeSepInN}
\end{prop}

We can now outline a two-step algorithm for constructing the DMEG from an 
arbitrary 
DMG, $\G$. We first construct the maximal Markov equivalent graph, 
$\mathcal{N}$. We know 
from Theorem 
\ref{thm:maxElm} that one can simply check if each pair of nodes are potential 
siblings/parents in the independence model induced by $\G$ and construct the 
maximal 
Markov equivalent graph directly. This may, however, not be computationally 
efficient.

The above propositions show that given the maximal DMG, one can efficiently 
construct the DMEG by 
evaluating separability once for each directed edge and 
twice for each bidirected edge. 
Using Proposition \ref{prop:noIPsep} one can determine separability by testing 
a 
single separation statement, and this 
means that starting from $\mathcal{N}$,  one can construct the 
corresponding DMEG 
in a way such that the number of 
separation 
statements to test scales linearly in the 
number of edges in $\mathcal{N}$.

		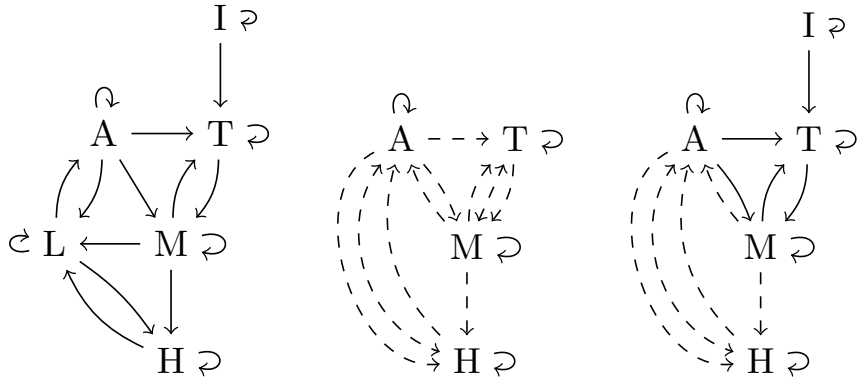
\begin{figure}
			\resizebox{0.3\textwidth}{!}{
				\begin{tikzpicture}[every node/.style={scale = 1}, every 
				loop/.style={min 
					distance=1mm,in=60,out=90,looseness=10}] 
				\node (A) {A};
				\node (T) [right of=A] {T};
				\node (M) [below right = .5cm and .075cm of A] {M};
				\node (H) [below of=M] {H};
				\node (L) [left of=M] {L};
				\node (I) [above of=T] {I};
				\path
				(A) [->]edge node {} (T)
				(T) [->]edge [bend left=20] node {} (M)
				(L) [->]edge [bend left = 20] node {} (A)
				(A) [->]edge [bend left = 20] (L)
				(A) [->]edge node {} (M)
				(M) [->]edge node {} (L)
				(M) [->]edge node {} (H)
				(L) [->]edge [bend left = 10] node {} (H)
				(H) [->]edge [bend left = 20] node {} (L)
				(I) [->]edge node {} (T)
				(M) [->]edge [bend left=20] node {} (T)
				(A) [->] edge [loop above, looseness = 5] (A)
				(I) [->] edge [loop right, looseness = 5] (I)
				(T) [->] edge [loop right, looseness = 5] (T)
				(L) [->] edge [loop left, looseness = 5] (L)
				(M) [->] edge [loop right, looseness = 5] (M)
				(H) [->] edge [loop right, looseness = 5] (H)
				;
				\end{tikzpicture}}
			\resizebox{0.3\textwidth}{!}{
				\begin{tikzpicture}[every node/.style={scale = 1}, every 
				loop/.style={min 
					distance=1mm,in=60,out=90,looseness=10}] 
				\node (A) {A};
				\node (T) [right of=A] {T};
				\node (M) [below right = .5cm and .075cm of A] {M};
				\node (H) [below of=M] {H};
				\node (L) [left of=M, fill opacity = 0] {L};
				\node (I) [above of=T, fill opacity = 0] {I};
				\path
				(A) [->]edge [dashed] node {} (T)
				(T) [->]edge [bend left=20, dashed] node {} (M)
				(L) [->]edge [bend left = 20, draw opacity = 0, dashed] node {} 
				(A)
				(A) [->]edge [bend left = 20, draw opacity = 0, dashed] (L)
				(A) [->]edge [bend left = 10, dashed] node {} (M)
				(M) [->]edge [bend left = 10, dashed] node {} (A)
				(A) [->]edge [bend right = 80, dashed] node {} (H)
				(H) [->]edge [bend left = 30, dashed] node {} (A)
				(A) [<->]edge [bend right = 55, dashed] node {} (H)
				(M) [->]edge [draw opacity = 0, dashed] node {} (L)
				(M) [->]edge [dashed] node {} (H)
				(M) [<->]edge [dashed] node {} (T)
				(L) [->]edge [bend left = 10, draw opacity = 0, dashed] node {} 
				(H)
				(H) [->]edge [bend left = 10, draw opacity = 0, dashed] node {} 
				(L)
				(I) [->]edge [draw opacity = 0, dashed] node {} (T)
				(M) [->]edge [bend left=20, dashed] node {} (T)
				(A) [->] edge [loop above, looseness = 5] (A)
				(I) [->] edge [loop right, looseness = 5, draw opacity = 0] (I)
				(T) [->] edge [loop right, looseness = 5] (T)
				(L) [->] edge [loop left, looseness = 5, draw opacity = 0] (L)
				(M) [->] edge [loop right, looseness = 5] (M)
				(H) [->] edge [loop right, looseness = 5] (H)
				;
				\end{tikzpicture}}
			\resizebox{0.3\textwidth}{!}{
				\begin{tikzpicture}[every node/.style={scale = 1}, every 
				loop/.style={min 
					distance=1mm,in=60,out=90,looseness=10}] 
					\node (A) {A};
					\node (T) [right of=A] {T};
					\node (M) [below right = .5cm and .075cm of A] {M};
					\node (H) [below of=M] {H};
					\node (L) [left of=M, fill opacity = 0] {L};
					\node (I) [above of=T] {I};
					\path
					(A) [->]edge node {} (T)
					(T) [->]edge [bend left=20] node {} (M)
					(L) [->]edge [bend left = 20, draw opacity = 0, dashed] 
					node {} 
					(A)
					(A) [->]edge [bend left = 20, draw opacity = 0, dashed] (L)
					(A) [->]edge [bend left = 10] node {} (M)
					(M) [->]edge [bend left = 10, dashed] node {} (A)
					(A) [->]edge [bend right = 80, dashed] node {} (H)
					(H) [->]edge [bend left = 30, dashed] node {} (A)
					(A) [<->]edge [bend right = 55, dashed] node {} (H)
					(M) [->]edge [draw opacity = 0, dashed] node {} (L)
					(M) [->]edge [dashed] node {} (H)
					(L) [->]edge [bend left = 10, draw opacity = 0, dashed] 
					node {} 
					(H)
					(H) [->]edge [bend left = 10, draw opacity = 0, dashed] 
					node {} 
					(L)
					(I) [->]edge  node {} (T)
					(M) [->]edge [bend left=20] node {} (T)
					(A) [->] edge [loop above, looseness = 5] (A)
					(I) [->] edge [loop right, looseness = 5] 
					(I)
					(T) [->] edge [loop right, looseness = 5] (T)
					(L) [->] edge [loop left, looseness = 5, draw opacity = 0] 
					(L)
					(M) [->] edge [loop right, looseness = 5] (M)
					(H) [->] edge [loop right, looseness = 5] (H)
					;
				\end{tikzpicture}}
			\caption{Left: Local independence graph of Example \ref{exmp:gw}. 
			Middle: 
			DMEG for the margina\-li\-zation over $L$ and
                        $I$. Right: DMEG for the margina\-li\-zation over $L$. 
			We have omitted the bidirected loops from the DMEGs and presented 
			the directed loops as solid.}
			\label{fig:gwDMG}
		\end{figure}

\begin{exmp}[Gateway drugs, continued]
	We return to the model in Example \ref{exmp:gw} to consider what happens 
	when it is only partially observed and to give an interpretation 
	of the corresponding local independence model. The local
        independence graph is assumed to be as depicted on Figure
        \ref{fig:gwDMG}, left. 
	
	Consider first the situation where $L$ and $I$ are
        unobserved. In this case, under the faithfulness assumption of the full 
        model (Definition
       \ref{def:faith}) we can construct
        the DMEG, which is shown in the center panel of Figure
        \ref{fig:gwDMG}, from the local independence model. The DMEG
        represents the Markov equivalence class which we can infer from
        the marginal local independence model ($L$ and $I$ are
        unobserved). Theoretically, the inference requires an oracle to
        provide us with local independence statements, which will in
        practice have to be approximated by statistical tests. What is
        noteworthy is that the DMEG can be inferred from the
        distribution of the observed variables only, and we do not
        need to know the local independences of the full model. 

        If we ignore which edges are dashed and
        which are not, the graph simply represents the local
        independence model of the marginal system as the maximal
        element in the Markov equivalence class. The
        dashed edges give us additional -- and in some sense local --
        information. As an example, the directed edge from $A$ to $H$ is
        dashed and we cannot know if there exists a conditioning set
        that would render $H$ locally independent of $A$ in the full
        system. On the other hand, the directed edge from 
        $T$ to $H$ is
        absent, and we can conclude that tobacco use is not directly affecting
        hard drug use.
	
	Consider instead the situation where $I$ is also observed. $I$ serves as an 
	analogue to an instrumental variable (see e.g. \cite{pearl2009} for an 
	introduction to instrumental variables). The 
	inclusion of this variable identifies some of the structure by removing 
	some dashed edges and making others non-dashed.		
\end{exmp}

\section{Discussion and conclusion}
\label{sec:conc}

In this paper we introduced a class of graphs to represent local 
independence structures of partially observed multivariate stochastic 
processes. Previous work 
based on 
directed graphs, that allows for cycles and use the asymmetric 
$\delta$-separation criterion,
was extended to mixed directed graphs to account for latent processes and we 
introduced $\mu$-separation in mixed directed
graphs.

An important task is the characterization of equivalence classes of graphs and 
this has been studied for e.g. MAGs \cite{ali2009, zhao2005}. In the case of 
MAGs, a key result 
is that every element in a Markov equivalence class has the same {\it 
	skeleton}, i.e. the same adjacencies \cite{ali2009}. As shown by 
Propositions 
\ref{prop:biIP} 
and \ref{prop:asymIP} this is not the case for DMGs, and Example 
\ref{exmp:nonMax} shows that one cannot necessarily within a Markov equivalence 
class find an element such that two nodes are inseparable if and only if 
they are adjacent.

We proved instead a central maximality property 
which allowed us to propose the use of DMEGs to represent a Markov 
equivalence class of DMGs in a concise 
way. Given a maximal DMG, 
we furthermore argued that one can efficiently find 
the DMEG. Similar results are known for chain graphs, as one can also 
in a certain sense find a unique, largest graph represen\-ting a Markov 
equivalence 
class \cite{frydenberg1990}, though this graph is not a supergraph of all 
Markov equivalent graphs 
as 
in the case of DMGs. Volf and Studen{\'y} \cite{volf1999} suggested to use 
this largest graph as 
a unique representative of the Markov equivalence class, and they provided an 
algorithm to construct it. 

We emphasize that the characterization given of the maximal element of a Markov 
equivalence class of DMGs
is constructive in the sense that it straightforwardly defines an algorithm 
for 
learning a maximal DMG from a local independence oracle. This learning 
algorithm may not be computationally efficient or even
feasible for large graphs, and it is ongoing research to develop 
efficient learning algorithms and to develop the practical
implementations of the tools needed for replacing the oracle by
statistical tests.

\section*{Acknowledgments}

This work was supported by a research grant from VILLUM FONDEN (13358). The 
authors are grateful to Steffen Lauritzen for his helpful comments and 
suggestions. We also thank two referees and an area editor whose comments have 
helped improve this manuscript.

%%%%%%%%%%%%%%%%%%%%%%%%%%%%%%%%%%%%
%%%%%%%%%%%%%%%%%%%%%%%%%%%%%%%%%%%% supp

\pagebreak
\setcounter{section}{0}
\renewcommand\thesection{\Alph{section}}

\title{Markov equivalence of marginalized local 
	independence graphs}
\begin{center}
	Supplementary material
\end{center}

\vspace*{1cm}

			In this supplementary material we discuss
			relations between $\mu$-separation and other asymmetric 
			notions of graphical separation. We also compare our 
			proposed definition of local independence to previous 
			definitions to argue that ours is in fact a 
			generalization. We furthermore relate $\mu$-separation
			to $m$-separation. We provide, in particular,
			a detailed discussion of the local
			independence model for discrete-time stochastic
			processes (time series), and we show how to
			verify $\mu$-separation via separation in an
			auxiliary undirected graph. We also
			discuss the existence of the compensators that
			are used in the definition of local
			independence for continuous-time stochastic
			process models. This supplementary material also 
			contains proofs of the results of the main paper.

			\section{Relation to other asymmetric notions of graphical 
				separation}
			\label{sec:otherSeps}
			
			In this section we relate $\mu$-separation to $\delta$-separation as
			introduced previously in the literature for directed graphs.

			\begin{defn}[Bereaved graph]
				Let $\mathcal{G} = (V,E_d)$ be a DG, and let $B\subseteq V$. 
				The 
				{\it $B$-bereaved graph},
				$\mathcal{G}^{B}$, is constructed from $\mathcal{G}$ by
				removing every directed edge with a tail at a node in $B$ 
				except loops. More
				precisely, $\mathcal{G}^{B} = (V, \bar{E}_d^B)$,
				where $\bar{E}_d^B = E_d \setminus \left(\bigcup_{\beta\in B} 
				\{(
				\beta, \delta
				) \mid \delta \neq \beta \}\right)$.
				\label{def:bereavedG}
			\end{defn}
			
			Didelez \cite{didelez2000} considered a DG, and for disjoint 
			sets 
			$A,B,C\subseteq V$ said that $B$ is 
			separated 
			from $A$ by $C$ if there is no $\mu$-connecting walk in 
			$\mathcal{G}^B$, or 
			equivalently, no $\mu$-connecting path. This 
			is called 
			$\delta$-separation. Note that the condition in Definitions 
			\ref{def:muConnSeq} 
			and \ref{def:muSep} that 
			a 
			connecting walk be nontrivial makes no difference now due to $A$ 
			and $B$ being 
			disjoint. The condition that a $\mu$-connecting walk ends with a 
			head at 
			$\beta\in B$ is also obsolete as we are evaluating separation in 
			the bereaved 
			graph $\G^B$. Didelez \cite{didelez2000} always assumed that a 
			process depended 
			on its own past, and thus did not visualize loops in the DGs as a 
			loop would 
			always be present at every node.
			
			Meek \cite{meek2014} generalized $\delta$-separation to 
			$\delta^*$-separation 
			in a 
			DG (allowing for loops)
			by considering only nontrivial $\mu$-connecting walks in 
			$\mathcal{G}^B$ for sets $A,B,C\subseteq V$ such that $A\cap C = 
			\emptyset$ 
			with the motivation that a 
			node can be separated from itself using this notion of separation. 
			However, 
			if we consider the
			graph $\alpha \rightarrow \beta$, and sets
			$A = \{\alpha\}$, $B = \{\alpha, \beta\}$, $C = \emptyset$, then 
			using 
			$\delta^*$-separation, $B$ is separated from $A$ given $C$, which 
			runs 
			counter to an intuitive understanding of separation. More 
			importantly,
			$\delta^*$-separation in the local independence graph will not
			generally imply local independence. 
			
			To establish an exact relationship between $\delta$- and 
			$\mu$-separations 
			and argue 
			that we are indeed proposing a generalization of the former, assume 
			that $\G$ 
			is 
			a DG and that $A,B,C 
			\subseteq V$ are disjoint. We will argue that
			
			\begin{equation}
			\musepG{A}{B}{C\cup B}{\G} \Leftrightarrow 
			\deltasepG{A}{B}{C}{\G}.
			\label{eq:deltaMuEqui}
			\end{equation}
			
			\noindent To see that this is the case, consider first a 
			$\delta$-connecting walk from $\alpha \in A$ to $\beta \in B$ given 
			$C$ in 
			$\mathcal{G}^B$, $\omega$. The subwalk from $\alpha$ to the first 
			node on 
			$\omega$ which is in $B$ 
			is also present and $\mu$-connecting given $C \cup B$ in 
			$\mathcal{G}$. On the other hand, assume that there 
			exists a $\mu$-connecting sequence, $\omega$, in $\mathcal{G}$. We 
			know that 
			$A\cap B = \emptyset$, and because $B$ is a 
			subset of 
			the conditioning set on the left hand side in 
			(\ref{eq:deltaMuEqui}), we 
			must have that the first time the path enters $B$, 
			it has a head at the node in $B$, and this implies that a subwalk 
			of $\omega$  
			is $\delta$-connecting, that is, present and connecting in $\G^B$. 
			In Section 
			\ref{sec:examp} we will discuss why $B$ 
			is included in the conditioning set on the left side of 
			(\ref{eq:deltaMuEqui}).

			\section{Markov properties}
			\label{sec:examp}
			
			The equivalence of pairwise and global Markov properties is pivotal 
			in much of 
			graphical modeling. In this section, we will show how our proposed 
			graphical 
			framework fits with known results on Markov properties in the 
			case of point processes and argue that our graphical framework is a 
			generalization 
			of that of 
			Didelez \cite{Didelez2008} to allow for non-disjoint sets and 
			unobserved 
			processes.
			
			\begin{defn}[The pairwise Markov property]
				Let $\I$ be 
				an independence model over $V$. We say that 
				$\I$ satisfies the {\it pairwise Markov property} with respect 
				to 
				the DG $\mathcal{D}$ if 
				for all $\alpha, 
				\beta\in V$,
				
				$$\alpha \not\rightarrow_\mathcal{D} \beta \Rightarrow 
				\ind{\alpha}{\beta}{V\setminus \{\alpha\}} \in \I.$$
				\label{def:pairwiseMarkov}
			\end{defn}

			\begin{defn}[The global Markov property]
				Let $A,B,C \subseteq V$. Let $\I$ be an 
				independence model over $V$. We 
				say that $\I$ satisfies the {\it global Markov property} with 
				respect 
				to the DMG 
				$\G$ if 
				$\I(\G) \subseteq \I$, i.e., if
				
				$$\musepG{A}{B}{C}{\G} \Rightarrow \ind{A}{B}{C} \in \I.$$
				\label{def:globalMarkov}
			\end{defn} 
			
			Didelez \cite{Didelez2008} only considered disjoint sets and gave a 
			slighty 
			different 
			definition of local independence. For disjoint sets, Didelez 
			\cite{Didelez2008} 
			defined
			that $B$ is locally independent of $A$ given $C$ if 
			
			$$
			A \not\rightarrow B \mid C \cup B,
			$$
			
			\noindent and we will make the relation between the two definitions 
			precise in 
			this section. Consider sets $\mathcal{S},\mathcal{S}_d \subseteq 
			\mathcal{P}(V) \times \mathcal{P}(V) \times \mathcal{P}(V)$, 
			
			\begin{align*}
			\mathcal{S}_d & = 
			\{(A, B, C) \mid A,B,C \text{ disjoint}, A,B \text{ non-empty} \}\\
			\mathcal{S} & = 
			\{(A, B, C) \mid B\subseteq C,\  A,C \text{ disjoint}, A,B \text{ 
				non-empty} \} 
			\end{align*}
			
			\noindent and the bijection $s: \mathcal{S}_d \rightarrow 
			\mathcal{S}$, 
			$s((A,B,C)) = 
			(A,B, 
			C\cup B)$. We will in this section let $\I$ denote a subset of
			$\mathcal{S}$ and let $\I^d$ denote a subset of
			$\mathcal{S}_d$. In Section \ref{sec:otherSeps} we argued that for 
			any 
			directed graph $\G$ and $(A,B,C) \in \mathcal{S}_d$,
			
			\begin{align*}
			\deltasepG{A}{B}{C}{\G} & \Leftrightarrow \musepG{A}{B}{C\cup B}{\G}
			\end{align*}
			
			\noindent and therefore
			
			$$
			\{ (A,B,C) \in \mathcal{S}_d : \deltasepG{A}{B}{C}{\G}  \} = s^{-1} 
			\Big( \{ 
			(A,B,C) 
			\in \mathcal{S} : \musepG{A}{B}{C}{\G}  \} \Big).
			$$
			
			\noindent For any local independence model defined by Didelez's 
			definition, 
			$\I^d$, and 
			any local independence model defined by Definition 
			\ref{def:localIndep}, $\I$, it holds that
			
			\begin{align*}
			\langle A,B\mid C\rangle \in \I^d & \Leftrightarrow A\not 
			\rightarrow B\mid 
			C\cup B \\ & \Leftrightarrow \langle A,B\mid C\cup B\rangle 
			\in \I
			\end{align*}
			
			\noindent so $
			\I^d = s^{-1}(\I)
			$. Hence, there is a bijection between the two sets, and graphical 
			and 
			probabilistic independence models are preserved under the 
			bijection. This means 
			that we have equivalence of Markov properties between 
			the two formulations. Thus, restricting our framework to 
			$\mathcal{S}$, we 
			get the equivalence of pairwise and global 
			Markov property directly from the proof by Didelez in the case of 
			point process 
			models, and we see that our seemingly dif\-fe\-rent definitions of 
			local 
			independence and graphical separation indeed give an extension of 
			earlier 
			work.
			
			One can show that for two DMGs $\G_1$, $\G_2$, that both have all 
			directed 
			and bidirected loops it holds that
			
			$$\I(\G_1) \cap \mathcal{S} = \I(\G_2) \cap \mathcal{S} 
			\Leftrightarrow \I(\G_1) = \I(\G_2).$$
			
			Let $\mathbb{G}$ denote the class of DMGs such that 
			all directed 
			and bidirected loops are present. Consider now some $\mathcal{G} 
			\in 
			\mathbb{G}$. By the above result we can identify the Markov 
			equivalence class 
			from the independence model restricted to $\mathcal{S}$. This 
			equivalence class 
			has a maximal 
			element which is also in $\mathbb{G}$ and thus one can also in this 
			case 
			represent the Markov equivalence class using a DMEG.

			\section{Time series and unrolled graphs}
			\label{supp:ts}
			
			In this section we first relate the cyclic DGs and 
			DMGs to acyclic 
			graphs and then 
			use this to discuss Markov properties (see Definition 
			\ref{def:globalMarkov}) and faithfulness of local 
			independence 
			models in the time series case.
			
			\begin{defn}[$m$-separation \cite{richardson2003}]
				Let $\mathcal{G} = (V,E)$ be a DMG and let $\alpha,\beta\in V$. 
				A
				path between 
				$\alpha$ and 
				$\beta$ is said to be $m$-connecting if no noncollider on the 
				path is in 
				$C$ 
				and every collider on the path is in 
				$An(C)$.  For 
				disjoint sets $A,B,C\subseteq V$, we say that $A$ and 
				$B$ 
				are $m$-separated by $C$ if there is no $m$-connecting path 
				between 
				$\alpha\in 
				A$ and $\beta\in B$.
				In this case, we write $\msep{A}{B}{C}$.
			\end{defn}
			
			The above $m$-separation is a generalization
			of the well-known $d$-separation in DAGs. In
			this section we will only consider
			$m$-separation for DAGs, and will thus use the
			$d$-separation terminology. In Section \ref{supp:aug} 
			we 
			provide a more general relation between 
			$\mu$-separation and $m$-separation. 
			
			We first describe how to obtain a DAG from a DG
			such that the DAG, if read 
			the right way, will 
			give the same separation model as the DG. 
			This can be 
			useful in time series 
			examples as well as when working with continuous-time models. 
			Sokol and Hansen \cite{sokol2014} studied solutions to 
			sto\-cha\-stic 
			differential 
			equations and used a DAG in discrete time to approximate the 
			continuous-time dynamics. Danks and Plis \cite{danks2013} and 
			Hyttinen et al. 
			\cite{danks2016} used similar 
			translations between an {\it 
				unrolled} graph in which time 
			is 
			discrete and explicit and a {\it rolled} graph in which time is 
			implicit. Some authors use the term {\it unfolded} instead of 
			unrolled.  In 
			a rolled graph each node represents a stochastic process whereas in 
			an unrolled 
			graph each node represents a single random variable. Definition 
			\ref{def:LIGtoDAG} 
			shows 
			how to unroll a local independence graph and Lemma 
			\ref{lem:rolledOutSeps} 
			establishes a precise relationship between
			independence models in the rolled and unrolled 
			graphs.
			
			\begin{defn}
				Let $\mathcal{G} = (V,E)$ be a DG and 
				let $T \in \mathbb{N}$. The unrolled version of $\mathcal{G}$, 
				$\mathcal{D}_T(\mathcal{G}) = 
				(\bar{V}, \bar{E})$,
				is the DAG on nodes 
				$$
				\bar{V} = \{x_t^\alpha \mid (t,\alpha) \in \{0,1,\ldots , T\} 
				\times 
				V\}
				$$
				and with edges
				$$
				\bar{E} = \{x_s^\alpha \rightarrow x_t^\beta\mid \alpha 
				\rightarrow_\mathcal{G} \beta \text{ and } s < t\}.
				$$
				Let $D \subseteq 
				V$ 
				and let $T\in \mathbb{N}$. We define $D_{0:T} = \{ 
				x_t^\alpha \in\bar{V} \mid \alpha \in D,\ t \leq T\}$ and $D_T 
				= \{ 
				x_t^\alpha 
				\in\bar{V} \mid \alpha \in D,\ t = T\}$.
				
				\label{def:LIGtoDAG}
			\end{defn}
			
			\begin{center}
				\begin{figure}
					\resizebox{\textwidth}{!}{%
						\begin{tikzpicture}[scale = 1, every 
						node/.style={circle, inner 
							sep = .2mm}, 
						font=\larger, bend angle = 15]
						
						% picture 1
						\node (a) {$\alpha$};
						\node (b) [below of=a] {$\beta$};
						\node (c) [below of=b] {$\gamma$};
						\path
						(a) [<-]edge (b)
						(b) [<-]edge node {} (c)
						(a) [->]edge [loop left] node {} (a)
						(b) [->]edge [loop right] node {} (b);
						
						% picture 2
						\begin{scope}[xshift=2cm]
						\node (a1) {$x_0^\alpha$};
						\node (b1) [below of=a1] {$x_0^\beta$};
						\node (c1) [below of=b1] {$x_0^\gamma$};
						\node (a2) [right of =a1] {$x_1^\alpha$};
						\node (b2) [below of=a2] {$x_1^\beta$};
						\node (c2) [below of=b2] {$x_1^\gamma$};
						\node (a3) [right of =a2] {$x_2^\alpha$};
						\node (b3) [below of=a3] {$x_2^\beta$};
						\node (c3) [below of=b3] {$x_2^\gamma$};
						\node (a4) [right of =a3]{$x_3^\alpha$};
						\node (b4) [below of=a4] {$x_3^\beta$};
						\node (c4) [below of=b4] {$x_3^\gamma$};
						\path
						(a1) [->] edge (a2)
						(a1) [->] edge[bend left=20] (a3)
						(a1) [->] edge[bend left=30] (a4)
						(a2) [->] edge (a3)
						(a2) [->] edge[bend left=20] (a4)
						(a3) [->] edge (a4)
						(b1) [->] edge (b2)
						(b1) [->] edge[bend left=20] (b3)
						(b1) [->] edge[bend left=30] (b4)
						(b2) [->] edge (b3)
						(b2) [->] edge[bend left=20] (b4)
						(b3) [->] edge (b4)
						(b1) [->] edge (a2)
						(b1) [->] edge (a3)
						(b1) [->] edge (a4)
						(b2) [->] edge (a3)
						(b2) [->] edge (a4)
						(b3) [->] edge (a4)	
						(c1) [->] edge (b2)
						(c1) [->] edge (b3)
						(c1) [->] edge (b4)
						(c2) [->] edge (b3)
						(c2) [->] edge (b4)
						(c3) [->] edge (b4)					
						;
						\end{scope}
						\end{tikzpicture}
						\label{fig:unrolled}
					}
					\caption{A directed graph (left) and the corresponding 
						unrolled version 
						with four time points, $\mathcal{D}_3(\G)$, (right). 
						$x_t^\delta$ 
						denotes the $\delta$-coordinate 
						process at time $t$ for $\delta \in 
						\{\alpha,\beta,\gamma\}$.}
				\end{figure}
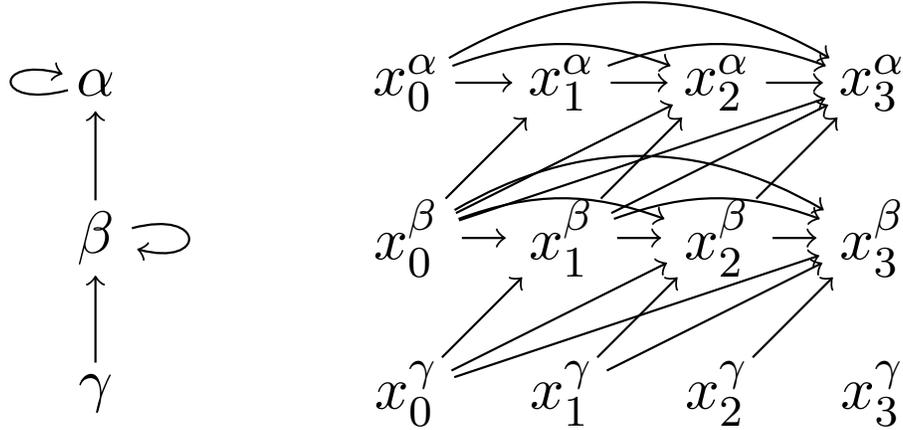
			\end{center}
			
			\begin{lem}
				Let $\G = (V,E)$ be a DG. If $\musepG{A}{B}{C}{\G}$ then 
				$\dsepG{(A\setminus 
					C)_{0:(T-1)}}{B_T}{C_{0:(T-1)}}{\mathcal{D}_T(\mathcal{G})}$.
				For large 
				enough values of $T$, the opposite implication holds as well.
				\label{lem:rolledOutSeps}
			\end{lem}
			
			\begin{proof}
				Assume first that  $\langle x_{s_0}^{\alpha_0}, e_1, 
				x_{s_1}^{\alpha_1} ,\ldots, e_l, 
				x_{s_l}^{\alpha_l} \rangle$ is a $d$-connecting path in 
				$\mathcal{D}_T(\mathcal{G})$. This path has a head at 
				$x_{s_l}^{\alpha_l} 
				\in B_T$. 
				Construct a 
				walk in $\mathcal{G}$ by for each node, $x_{s_k}^{\alpha_k}$, 
				taking the corresponding 
				node, $\alpha_k$, and for each edge $x_{s_k}^{\alpha_k} \sim 
				x_{s_{k+1}}^{\alpha_{k+1}}$ taking the corresponding, 
				endpoint-identical 
				edge $\alpha_{k} \sim 
				\alpha_{k+1}$ in 
				$\mathcal{G}$. On this walk, no noncollider is in $C$, and 
				every 
				collider is 
				an ancestor of a node in $C$.
				
				Assume instead that $\omega$ is a $\mu$-connecting walk in 
				$\mathcal{G}$ 
				from $A$ to $B$ 
				given $C$,
				
				$$
				\alpha_1 \sim \ldots \sim \alpha_{l-1} \rightarrow \alpha_l
				$$
				
				\noindent and let $T \geq 3(\vert E\vert + 1) + 1$. Using 
				Proposition 
				\ref{prop:muConnPath}, we can assume that $\omega$ has length 
				smaller than 
				or equal to $\vert E\vert + 1$. We construct 
				a $d$-connecting walk in $\mathcal{D}_T(\G)$
				in 
				the following way. Starting from $x_T^{\alpha_l}$, we choose 
				the edge 
				between 
				$x_{\vert E\vert + 1}^{\alpha_{l-1}}$ and $x_T^{\alpha_l}$. For 
				the 
				remaining 
				edges, $\alpha_k \sim 
				\alpha_{k+1}$, we choose the edge $x_{s_k-1}^{\alpha_k} 
				\rightarrow 
				x_{s_k}^{\alpha_{k+1}}$ if $\alpha_k \rightarrow 
				\alpha_{k+1}$ in $\omega$, and $x_{s_k}^{\alpha_{k + 1}} 
				\rightarrow 
				x_{s_k +1 }^{\alpha_{k}}$ if $\alpha_k \leftarrow 
				\alpha_{k+1}$ in $\omega$ where $s_k$ is determined by the 
				endpoints of the 
				previous 
				edge. 
				No 
				noncollider on 
				this walk 
				will be in $C_{0:(T-1)}$. Every collider will be in 
				$An_{\mathcal{D}_T(\mathcal{G})}(C_{0:(T-1)})$ as the collider 
				will be in 
				the time 
				slices $0$ 
				to 
				$2(\vert E\vert + 1)$.  This $d$-connecting walk can be trimmed 
				down to a 
				$d$-connecting path.
			\end{proof}
			
			We defined local
			independence for a class of continuous-time
			processes in Definition \ref{def:localIndep}. In this 
			section we define a similar
			notion for time series, as also introduced in 
			\cite{eichler2010}. Let $V 
			= 
			\{1,\ldots,n\}$. We consider a 
			multivariate time series $(X_t)_{t\in \mathbb{N}\cup {\{0\}}}$, $ 
			X_t = 
			(X_t^1,\ldots,X_t^n)$, of the form
			
			$$
			X_t^\alpha = f_{\alpha t}(X_{s<t}, \varepsilon_t^\alpha),
			$$
			
			\noindent where $X_{s<t} = \{X_u^\alpha \mid \alpha \in V,u < t 
			\}$. 
			The random variables $\{\varepsilon_t^\alpha\}$ are independent. 
			For 
			$S\subseteq 
			\mathbb{N} \cup \{0\}$ and
			$D\subseteq V$ 
			we 
			let 
			$X_S^D = \{X_s^\alpha \mid \alpha 
			\in 
			D, s \in S  \}$ and $X^D = \{ X^\alpha \mid \alpha \in D\}$. In the 
			case of 
			time 
			series, 
			a notable feature of local independence 
			and local 
			independence 
			graphs is that they provide a simple representation in comparison 
			with graphs 
			in which each vertex represents a single time-point variable. 
			
			\begin{defn}[Local independence, time series]
				Let $X$ be a multivariate time series. We say that $X^B$ is 
				{\it locally 
					independent} of $X^A$ given $X^C$ if for all $t \in 
				\mathbb{N}$, 
				$\beta 
				\in B$, $X_{s<t}^A$  and $X_t^\beta$ are conditionally 
				independent 
				given $ 
				X_{s<t}^C$, that is,
				
				$$
				X_{s<t}^A \perp \!\!\! \perp X_t^\beta \mid X_{s<t}^C
				$$
				
				\noindent and write 
				$A 
				\not\rightarrow B \mid C$.
			\end{defn}
			
			The above definition induces an independence
			model over $V$, which we will also refer to as
			the local independence model and denote
			$\mathcal{I}$ in the following. The main
			question that we address is whether this
			independence model is graphical.  That is, we
			will construct a DG, consider the Markov and 
			faithfulness
			properties of $\mathcal{I}$ and
			this DG, and relate them to Markov and
			faithfulness properties of the conditional
			independence model of finite distributions and
			unrolled versions of the DG.
			
			\begin{defn}[Faithfulness]
				Let $A,B,C \subseteq V$. Let $\I$ be an 
				independence model on $V$ and let $\G$
				be a DMG. We 
				say that $\I$ and $\G$ are {\it faithful} if 
				$\I = \I(\G)$, i.e., if
				
				$$\ind{A}{B}{C} \in \I \Leftrightarrow \musepG{A}{B}{C}{\G}.$$
				
				One can give analogous definitions using other notions of 
				graphical separation. Below we also consider faithfulness of a 
				probability distribution and a DAG, implicitly using 
				$d$-separation instead of $\mu$-separation in the above 
				definition.
				\label{def:faith}
			\end{defn} 
			
			Let 
			$\mathcal{D}_T$ for $T \geq 1$
			be the DAG on nodes $\{x_s^\alpha\mid s \in \{0,\ldots,T\}, 
			\alpha \in V \}$ such that there is an edge $x_s^\alpha \rightarrow 
			x_t^\beta$ 
			if and only if $f_{\beta t}$ depends on the argument
			$X_s^\alpha$. Let $D_S =  \{x_s^\alpha \mid
			\alpha\in D, s \in S\}$. Let $\G$ denote  
			the minimal 
			DG such 
			that its unrolled version, $\mathcal{D}_T(\G)$, is a supergraph of 
			$\mathcal{D}_T$ for all $T \in \mathbb{N}$.
			
			For all $T\in 
			\mathbb{N}$, the DAG 
			$\mathcal{D}_T(\mathcal{G})$ and the distribution of $X_{s\leq T}$ 
			satisfy 
			
			$$
			x_s^\alpha, x_t^\beta \text{ not adjacent } \Rightarrow X_s^\alpha 
			\perp \!\!\! 
			\perp 
			X_t^\beta \mid 
			(An(X_s^\alpha) \cup An(X_t^\beta)) \setminus 
			\{X_s^\alpha,X_t^\beta \},
			$$
			
			\noindent which is also known as the pairwise Markov property for 
			DAGs.
			Assume
			equivalence of the pairwise and global Markov properties for this 
			DAG and the
			finite-dimensional distribution (see e.g. \cite{lauritzen2017} for 
			necessary 
			and 
			sufficient conditions for this equivalence).
			Assume that $B$ is $\mu$-separated from $A$ by $C$ in the DG 
			$\mathcal{G}$, 
			$\musepG{A}{B}{C}{\G}$.  By Lemma 
			\ref{lem:rolledOutSeps},  
			$\msepG{(A\setminus 
				C)_{s<T}}{B_T}{C_{s<T}}{\mathcal{D}_T(\mathcal{G})}$, and 
			by the 
			global Markov property in this DAG, 
			$X_{s<T}^{A\setminus C} \perp \!\!\! \perp X_T^B \mid X_{s<T}^C$. 
			This 
			holds for any $T$, and therefore $A\setminus C \not\rightarrow B 
			\mid C$. It 
			follows that $A \not\rightarrow B \mid C$. This means that 
			$\mathcal{I}$ satisfies the global Markov property with respect to 
			$\G$.  
			
			Assume furthermore that the distribution of $X_T$ and the DAG 
			$\mathcal{D}_T(\mathcal{G})$ for some $T\in
			\mathbb{N}$ are faithful and  that 
			$T \geq 
			3(\vert E\vert +1) +1 $. Meek \cite{meek1995} 
			studied faithfulness of DAGs and argued that faithful 
			distributions 
			exist 
			for any DAG. 
			If $A \not \rightarrow B 
			\mid C$, then $A\setminus C \not\rightarrow B \mid C$ and 
			$X_{s<T}^{A\setminus 
				C} \perp \!\!\! \perp 
			X_T^B 
			\mid X_{s<T}^C$. By faithfulness of the 
			distribution of $X_T$ and the DAG $\mathcal{D}_T(\G)$, we have 
			$\msepG{(A\setminus C)_{s<T}}{B_T}{C_{s<T}}{\mathcal{D}_T(\G)}$ and 
			using 
			Lemma 
			\ref{lem:rolledOutSeps} this implies that $\musepG{A}{B}{C}{\G}$, 
			giving us 
			faithfulness of $\mathcal{I}$ and $\G$. 
			
			In summary, for every DG there exists a
			time series such that the local independence
			model induced by its distribution and the DG
			are faithful.
			
			\section{An augmentation criterion}
			\label{supp:aug}
			
			In this section we present results that allow us to determine 
			$\mu$-separation 
			from graphical separation in an undirected graph. An {\it 
				undirected graph} is 
			a graph, $(V,E)$, with an edge set that consists of 
			unordered 
			pairs of 
			nodes such that every edge is of the type $-$. Let $A,B$, and $C$ 
			be disjoint 
			subsets 
			of $V$. We say that $A$ and $B$ are separated by $C$ if every path 
			between 
			$\alpha\in A$ and $\beta\in B$ contains a node in $C$.
			
			When working with $d$-separation in DAGs, it is possible to give an 
			equivalent 
			se\-pa\-ration criterion using a derived undirected graph, the {\it 
				moral} 
			graph 
			\cite{lauritzen1996}. Didelez \cite{didelez2000} also gives both
			pathwise and so-called moral graph criteria for 
			$\delta$-separation. The 
			augmented graph below is a generalization of the moral graph 
			\cite{ richardson2003, richardson2002} which allows one to give a 
			criterion for 
			$m$-separation based on an augmented graph. We use the 
			similarity of $\mu$-separation and $m$-separation to give an 
			augmentation graph 
			criterion for $\mu$-separation. The first step 
			in making a connection to $m$-separation is to 
			explicate that each node 
			of a DMG re\-pre\-sents an entire stochastic 
			process, and notably, both the past and the present of that 
			process. We do that 
			using graphs of the below type.
			
			\begin{defn}
				Let $\mathcal{G} = (V,E)$ and let $B = 
				\{\beta_1,\ldots,\beta_k\}\subseteq 
				V$. The 
				$B$-history 
				version of $\mathcal{G}$, denoted by $\mathcal{G}(B)$, is the 
				DMG with node 
				set $V \disjU \{\beta_1^p,\ldots, \beta_k^p\}$ such that 
				$\mathcal{G}(B)_V 
				= 
				\mathcal{G}$ and
				
				\begin{itemize}
					\item $\alpha \leftrightarrow_{\mathcal{G}(B)} \beta_i^p$ 
					if $\alpha 
					\leftrightarrow_{\mathcal{G}} \beta_i$ and $\alpha \in V, 
					\beta_i \in 
					B$, 
					\item $\alpha \rightarrow_{\mathcal{G}(B)} \beta_i^p$ if
					$\alpha \rightarrow_\mathcal{G} \beta_i$ and $\alpha \in V, 
					\beta_i \in 
					B$.
				\end{itemize}
				\label{def:historyVersion}
			\end{defn}
			
			$\mathcal{G}(B)$ is a graph such that every node $b\in B$ is simply 
			split in 
			two: 
			one that represents the present and one that represents the past. 
			We define 
			$B^p = \{\beta_1^p,\ldots, \beta_k^p\}$.
			
			\begin{prop}
				Let $\mathcal{G} = (V,E)$ be a DMG, and let $A,B,C \subseteq 
				V$. Then 
				
				$$
				\musepG{A}{B}{C}{\G} \Leftrightarrow \msepG{A\setminus 
					C}{B^p}{C}{\G(B)}.
				$$
				\label{prop:muConnHist}
			\end{prop}
			
			\begin{proof}
				Assume first that there is a $\mu$-connecting walk from 
				$\alpha\in A$ to 
				$\beta\in B$ given $C$ in $\mathcal{G}$. By definition $\alpha 
				\in 
				A\setminus C$. By Proposition 
				\ref{prop:muConnPath} there is a $\mu$-connecting route,
				
				$$
				\alpha \sim \ldots \sim \beta \sim \ldots \gamma  
				\starrightarrow \beta.
				$$
				
				The subwalk from $\alpha$ to $\gamma$ is also present in 
				$\mathcal{G}(B)$ and composing it with $\gamma  
				\starrightarrow_{\mathcal{G}(B)} \beta^p$ 
				gives an $m$-connecting path between $A\setminus C $ and $B^p$ 
				which is 
				open 
				given $C$.  
				
				On the other hand, if there is an $m$-connecting path from 
				$\alpha\in A 
				\setminus C$ 
				to 
				$\beta^p\in B^p$ given $C$ in $\mathcal{G}(B)$, then no 
				non-endpoint node 
				is in $B^p$,
				
				$$
				\alpha \sim \ldots \gamma \starrightarrow \beta^p
				$$
				
				The subpath from $\alpha$ to $\gamma$ is present in 
				$\mathcal{G}$ and can 
				be composed with the edge $\gamma \starrightarrow \beta$ to 
				obtain 
				a 
				$\mu$-connecting walk from $A$ to $B$ given $C$ in 
				$\mathcal{G}$.
			\end{proof}

			\begin{defn}
				Let $\mathcal{G} = (V,E)$ be a DMG. We define the {\it 
					augmented graph} of 
				$\mathcal{G}$, $\mathcal{G}^a$, to be the undirected graph 
				without loops 
				and with node set 
				$V$ such that two distinct nodes are adjacent if and only if 
				the two nodes 
				are 
				collider connected in $\mathcal{G}$.
			\end{defn}
			
			\begin{prop}
				Let $\mathcal{G} = (V,E)$ be a DMG, $A,B,C\subseteq V$. Then 
				$\musepG{A}{B}{C}{\G}$ if and only if 
				$A\setminus C$ and $B^p$ are separated by $C$ in the augmented 
				graph of 
				$\mathcal{G}(B)_{An(A\cup B^p\cup C)}$.
			\end{prop}
			
			\begin{proof}
				Using Proposition \ref{prop:muConnHist} we have that 
				$\musepG{A}{B}{C}{\G} 
				\Leftrightarrow \msepG{A\setminus C}{B^p}{C}{\G(B)}$. Let 
				$\G(B)'$ be the DMG obtained from $\G(B)$ by removing all 
				loops. Then 
				$\msepG{A\setminus C}{B^p}{C}{\G(B)}$  if and only if 
				$\msepG{A\setminus C}{B^p}{C}{\G(B)'}$. We can apply Theorem 
				1 of 
				\cite{richardson2003}. That theorem assumes an ADMG, however, 
				as noted in 
				the paper, acyclicity is not used in the proof 
				which therefore also applies to $\G(B)'$, and we conclude that 
				$\msepG{A\setminus C}{B^p}{C}{\G(B)'}$ if 
				and only if $A \setminus C$ and $B^p$ 
				are separated by $C$ in $(\G(B)'_{An(A\cup B^p\cup C)})^a = 
				(\G(B)_{An(A\cup B^p\cup C)})^a$.

			\end{proof}
			
			\section{Existence of compensators}
			\label{supp:compensators}
			
			Let $Z = (Z_t)$ denote a real-valued stochastic process defined on a
			probability space $(\Omega, \mathcal{F},
			P)$, and let $(\mathcal{G}_t)$ denote a
			right-continuous and complete filtration
			w.r.t. $P$ such that $\mathcal{G}_t \subseteq
			\mathcal{F}$. Note that $Z$ is not assumed
			adapted w.r.t. the filtration. When
			$Z$ is a right-continuous process of finite and integrable
			variation, it follows from Theorem
			VI.21.4 in \cite{rogers2000} that there exists a
			predictable process of integrable variation,
			$Z^p$, such that ${}^{\mathrm{o}}Z - Z^{\mathrm{p}}$ is 
			a
			martingale. Here ${}^{\mathrm{o}}Z$ denotes the
			\emph{optional projection} of $Z$, which is a
			right-continuous version of the process
			$(E(Z_t \mid \mathcal{G}_t))$, cf. Theorem
			VI.7.1 and Lemma VI.7.8 in
			\cite{rogers2000}. The process $
			\Lambda = Z^{\mathrm{p}}$ is called the dual
			predictable projection or compensator of the
			optional projection ${}^{\mathrm{o}}Z$ as well
			as of the process $Z$ itself. It depends
			on the filtration $(\mathcal{G}_t)$. 
			
			If $Z$ is adapted w.r.t. a (right-continuous and
			complete) filtration $(\mathcal{F}_t)$, it has a compensator
			$\tilde{\Lambda} = Z^{\mathrm{p}}$ such that
			$Z - \tilde{\Lambda}$ is an $\mathcal{F}_t$ martingale. 	
			When $\mathcal{G}_t \subseteq \mathcal{F}_t$ it may be of interest 
			to understand the
			relation between $\Lambda$, as defined above
			w.r.t. $(\mathcal{G}_t)$, and $\tilde{\Lambda}$.
			If $\tilde{\Lambda}$ is continuous with $\tilde{\Lambda}_0 = 0$, 
			say, we may ask
			if $\Lambda$ equals the \emph{predictable projection},
			$E(\tilde{\Lambda}_t \mid \mathcal{G}_{t-})$. As $\tilde{\Lambda}$  
			is assumed continuous and is of finite variation, 
			$$\tilde{\Lambda}_t = \int_0^t \tilde{\lambda}_s
			\mathrm{d} s.$$
			If $( \tilde{\lambda}_t)$ itself is an integrable right-continuous
			process, then its
			optional projection, $( E(\tilde{\lambda}_t \mid 
			\mathcal{G}_{t}))$, is
			an integrable right-continuous process, and 
			$$E(\tilde{\Lambda}_t \mid \mathcal{G}_{t-}) = \int_0^t 
			E(\tilde{\lambda}_s \mid
			\mathcal{G}_{s}) \mathrm{d} s $$ is a finite-variation, continuous
			version of the predictable projection of $\tilde{\Lambda}$. It is 
			clear that
			$$E(Z_t \mid \mathcal{G}_t) - \int_0^t E(\tilde{\lambda}_s \mid
			\mathcal{G}_{s}) \mathrm{d} s$$ is
			a $\mathcal{G}_t$ martingale, thus 
			$$\Lambda_t = \int_0^t E(\tilde{\lambda}_s \mid \mathcal{G}_{s}) 
			\mathrm{d} s$$
			is a compensator of $Z$ w.r.t. the filtration $(\mathcal{G}_t)$.
			
			We formulate the consequences of the discussion as a criterion
			for determining local independence via the computation of
			conditional expectations. The setup is as in Definition 
			\ref{def:localIndep} in Section \ref{sec:locind}.
			
			\begin{prop} Assume that the process $X^{\beta}$ for all $\beta \in 
				V$
				has a compensator w.r.t. the filtration $(\mathcal{F}^V_t)$ of 
				the form 
				$$\Lambda^{V,\beta}_t = \Lambda_0^{V,\beta} + \int_0^t 
				\lambda_s^{\beta} \mathrm{d}s$$
				for an integrable right-continuous process 
				$(\lambda_t^{\beta})$ and a
				deterministic constant $\Lambda_0^{V,\beta}$. Then $X^{\beta}$ 
				is
				locally independent of $X^A$ given $X^C$ for $A, C \subseteq V$ 
				if the
				optional projection 
				$$E(\lambda_t^{\beta} \mid \mathcal{F}^{A \cup C}_{t})$$
				has an $\mathcal{F}_t^{C}$ adapted version.
			\end{prop}
			
			Another way to phrase the conclusion of the proposition is that if 
			the optional
			projection $E(\lambda_t^{\beta} \mid \mathcal{F}^{C}_{t})$ is
			indistinguishable from $E(\lambda_t^{\beta} \mid \mathcal{F}^{A \cup
				C}_{t})$, then $A \not \rightarrow \beta \mid C$, and it is a 
			way
			of testing local independence via the computation of conditional
			expectations. It is a precise formulation of the innovation theorem
			stating how to compute compensators for one filtration via 
			conditional 
			expectations 
			of compensators for a superfiltration. 
			
			\section{Proofs} The following are proofs of the results from 
			the main paper.
			\label{sec:proofs}
			
			\begin{proof}[Proof of Proposition \ref{prop:muConnAllCollInC}]
				Let $\omega$ be a $\mu$-connecting walk given 
				$C$ and let $\gamma$ 
				be a collider on the walk such that $\gamma \in An(C) \setminus 
				C$. Then 
				there 
				exists a 
				subwalk $\bar{\omega} = \alpha_1 \starrightarrow \gamma 
				\starleftarrow 
				\alpha_2$, and an open (given $C$), directed path from $\gamma$ 
				to $\delta 
				\in 
				C$, 
				$\pi$. By composing  $\alpha_1 \starrightarrow \gamma$ with 
				$\pi$, 
				$\pi^{-1}$, 
				and $\gamma \starleftarrow 
				\alpha_2$ we get an open walk which is endpoint-identical to 
				$\bar{\omega}$ 
				and 
				with its only collider, $\delta$, in 
				$C$, and we can substitute $\bar{\omega}$ with this new walk. 
				Making 
				such 
				a substitution for every collider in $An(C) \setminus C$ on 
				$\omega$, we 
				obtain a $\mu$-connecting walk on which every collider is in 
				$C$.
			\end{proof}
			
			\begin{proof}[Proof of Proposition \ref{prop:muConnPath}]
				Assume that 
				we start from $\alpha$ and continue along $\omega$ until some 
				node, $\gamma 
				\neq \beta$, is 
				repeated. Remove the cycle from $\gamma$ to $\gamma$ 
				to obtain another walk from $\alpha$ to $\beta$, 
				$\bar{\omega}$. If 
				$\gamma= \alpha$, then 
				$\bar{\omega}$ is $\mu$-connecting. Instead assume $\gamma \neq 
				\alpha$. 
				If  this instance of $\gamma$ is a 
				noncollider on $\bar{\omega}$ then it must have been a 
				noncollider in an 
				instance on $\omega$ and thus $\gamma \notin C$. If on the 
				other hand this 
				instance of $\gamma$ is a collider on $\bar{\omega}$ then 
				either $\gamma$ 
				was a collider in an instance on $\omega$ or the ancestor of a 
				collider on 
				$\omega$, and thus $\gamma \in An(C)$. In either case, we see 
				that 
				$\bar{\omega}$ is a $\mu$-connecting walk. Repeating this 
				argument, we 
				can construct a $\mu$-connecting walk where only $\beta$ is 
				potentially 
				repeated. If there is $n > 2$ instances of $\beta$ then we can 
				remove 
				at least $n-2$ of them as above as long as we leave an edge 
				with a head at 
				the final $\beta$. 
			\end{proof}
			
			\begin{proof}[Proof of Proposition \ref{prop:musepDG}]
				Note first that a vertex can be a 
				parent 
				of 
				itself. The result then follows from the fact
				that
				$\musep{\alpha}{\beta}{\mathrm{pa}(\beta)}$.
			\end{proof}
			
			\begin{proof}[Proof of Proposition \ref{prop:closedMarg}]
				The first statement follows from the fact that no edge without 
				heads (i.e. 
				$-$) is ever added. Assume for the second statement that $\G$ 
				satisfies 
				(\ref{eq:mDGprop}). Let $M = V \backslash O$. Assume $\alpha 
				\leftrightarrow_\mathcal{M} \beta$. By definition of the latent 
				projection, 
				we can find an endpoint-identical walk between $\alpha$ and 
				$\beta$ in $\G$ 
				with no colliders and such that all non-endpoint nodes are in 
				$M$. 
				Either this walk has a bidirected edge at $\alpha$ in which 
				case $\alpha 
				\leftrightarrow_\G \alpha$ by (\ref{eq:mDGprop}) and therefore 
				also $\alpha 
				\leftrightarrow_\mathcal{M} \alpha$. Otherwise, there is a 
				directed edge 
				from some node $\gamma \in M$ such that $\gamma \rightarrow_\G 
				\alpha$. 
				Then the walk $\alpha \leftarrow \gamma \rightarrow \alpha$ is 
				present in 
				$\G$ and therefore $\alpha 
				\leftrightarrow_\mathcal{M} \alpha$ because $\mathcal{M}$ is a 
				latent 
				projection.
			\end{proof}
			
			\begin{proof}[Proof of Proposition \ref{prop:selfEdgeChar}]
				Assume first that $\alpha$ has no loops. In this case, there 
				are no 
				bidirected edges between $\alpha$ and any node, and therefore 
				the edges 
				that have a head at $\alpha$ have a tail at the previous node. 
				Any 
				nontrivial walk between $\alpha$ and $\alpha$ is 
				therefore blocked by $V\setminus \{\alpha\}$. Conversely, if 
				$\alpha$ has a 
				loop, then $\alpha 
				\starrightarrow \alpha$ is a $\mu$-connecting walk given 
				$V\setminus 
				\{\alpha\}$.
			\end{proof}
			
			\begin{proof}[Proof of Theorem \ref{thm:DMGmarg}] Let $M = V 
				\backslash O$. 
				Let first $\omega$ be a $\mu$-connecting walk from $\alpha\in 
				A$ to 
				$\beta\in B$ given $C$ in $\G$. Using Proposition 
				\ref{prop:muConnAllCollInC}, we can find a $\mu$-connecting 
				walk from 
				$\alpha\in A$ to $\beta\in B$ given $C$ in $\G$ such that all 
				colliders are 
				in $C$. Denote this walk by $\bar{\omega}$. Every node, $m$, on 
				$\bar{\omega}$ which is in $M$ is on a subwalk of 
				$\bar{\omega}$, $\delta_1 
				\sim \ldots \sim m \sim \ldots \sim \delta_2$, such that 
				$\delta_1,\delta_2\in O$ and all other nodes on the subwalk are 
				in $M$. 
				There are no 
				colliders on this subwalk and therefore there is an 
				endpoint-identical edge 
				$\delta_1 \sim \delta_2$ in $\mathcal{M}$. Substituting all 
				such subwalks 
				with their corresponding endpoint-identical edges gives a 
				$\mu$-connecting 
				walk in $\mathcal{M}$.
				
				On the other hand, let $\omega$ be a $\mu$-connecting walk from 
				$A$ to $B$ 
				given $C$ in $\mathcal{M}$. Consider some edge in $\omega$ 
				which is not in 
				$\G$. In $\G$ there is an endpoint-identical walk with no 
				colliders and no 
				non-endpoint nodes in $C$. Substituting each of these edges 
				with such an 
				endpoint-identical walk gives a $\mu$-connecting walk in $\G$ 
				using 
				Proposition \ref{prop:preserveAn}.
			\end{proof}
			
			\begin{proof}[Proof of Proposition \ref{prop:margAlgo}]
				We first note that in Algorithm \ref{algo:marg} adding an edge 
				will never 
				remove any triroutes. Therefore, Algorithm \ref{algo:marg} 
				returns 
				the same output regardless of the order in which the algorithm 
				adds edges. 
				
				Let $\mathcal{M}$ denote 
				the output of Algorithm \ref{algo:marg} which is clearly a DMG. 
				The graphs 
				$\mathcal{M}$ and $m(\G, O)$ have the same 
				node set, thus it suffices to show that also the edge sets are 
				equal. 
				Assume first $\alpha \overset{e}{\sim}_{m(\G, O)} \beta$. Then 
				there 
				exist an endpoint-identical walk in $\G$ that contains no 
				colliders and 
				such that all the non-endpoint nodes are in $M = V\setminus O$, 
				$\alpha 
				\sim \gamma_1 \sim 
				\ldots \sim \gamma_n \sim \gamma_{n+1} = \beta$. Let $e_l$ be 
				the edge 
				between $\alpha$ and $\gamma_l$ which is endpoint-identical to 
				the subwalk 
				from $\alpha$ to $\gamma_l$. If $e_l$ is present in 
				$\mathcal{M}_k$ at some 
				point during 
				Algorithm 
				\ref{algo:marg}, then edge $e_{l+1}$ will also be added before 
				the 
				algorithm terminates, $l = 1,\ldots,n$. We see that $e_1$ is in 
				$\G$, and 
				this means that $e$ 
				is 
				also present in $\mathcal{M}$.
				
				On the other hand, assume that some edge $e$ is in 
				$\mathcal{M}$. If $e$ is 
				not in $\G$, then we can find a noncolliding, 
				endpoint-identical 
				triroute in the graph $\mathcal{M}_k$ ($k$ has the value that 
				it takes when 
				the algorithm terminates) such that the noncollider is in $M$. 
				By 
				repeatedly using 
				this argument, we can from any edge, $e$, in $\mathcal{M}$ 
				construct
				an endpoint-identical walk in $\G$ that contains no colliders 
				and such that 
				every non-endpoint node is in $M$, and therefore $e$ is also 
				present in 
				$m(\G, O)$.
			\end{proof}
			
			\begin{proof}[Proof of Proposition \ref{prop:IPepiConnPath}]
				Let 
				
				$$
				\alpha\ *\!\! \rightarrow \gamma_1 \leftrightarrow \ldots 
				\leftrightarrow 
				\gamma_n \leftrightarrow \beta
				$$
				
				\noindent be the inducing path, $\nu$. Let $\gamma_{n+1}$ 
				denote $\beta$. 
				If $\nu$ has length one, then 
				it is directed or 
				bidirected and itself a $\mu$-connecting path/cycle regardless 
				of $C$. 
				Assume instead that 
				the length of $\nu$ is 
				strictly larger than one, and assume also first that $\alpha 
				\neq \beta$. 
				Let 
				$k$ be the maximal index in $\{1,\ldots,n\}$ such that there 
				exists an open 
				walk from $\alpha$ to $\gamma_k$ given $C$ which does not 
				contain $\beta$ 
				and only contains $\alpha$ once. 
				There is a $\mu$-connecting walk from 
				$\alpha$ to $\gamma_1 \neq \beta$ given $C$ and therefore $k$ 
				is always 
				well-defined.
				
				Let $\omega$ be the open walk from $\alpha$ to $\gamma_k$. If 
				$\gamma_k\in 
				An(C)$, then the 
				composition of $\omega$ with the edge $\gamma_k \leftrightarrow 
				\gamma_{k+1}$ 
				is open from $\alpha$ to $\gamma_{k+1}$ given $C$. By 
				maximality of $k$, we 
				must have $k = n$, and the composition is 
				therefore an open walk from $\alpha$ to $\beta$ on which 
				$\beta$ only 
				occurs once. We can reduce this to a $\mu$-connecting path 
				using arguments 
				like those in the 
				proof of Proposition \ref{prop:muConnPath}. Assume instead that 
				$\gamma_k\notin An(C)$. There is a directed path from 
				$\gamma_k$ to 
				$\alpha$ or to $\beta$. Let $\pi$ denote the subpath from 
				$\gamma_k$ to the 
				first occurrence of either $\alpha$ or $\beta$ on this directed 
				path. If 
				$\beta$ occurs first, 
				then the 
				composition of $\omega$ with $\pi$ gives an open walk from 
				$\alpha$ to 
				$\beta$. There is a head at 
				$\beta$ when moving from $\alpha$ to $\beta$ and therefore the 
				walk can be 
				reduced to a $\mu$-connecting path from $\alpha$ to $\beta$ 
				using the 
				arguments in the proof of Proposition \ref{prop:muConnPath}. If 
				$\alpha$ 
				occurs first, then the composition of $\pi^{-1}$ and the edge 
				$\gamma_k 
				\leftrightarrow \gamma_{k+1}$ gives a 
				$\mu$-connecting walk and it follows that $k = n$ by maximality 
				of 
				$k$. This walk is a $\mu$-connecting path.
				
				To argue that the open path is endpoint-identical if $\nu$ is 
				directed or 
				bidirected, let instead $k$ be the 
				maximal index such that there exists a $\mu$-connecting 
				walk from $\alpha$ to $\gamma_k$ with a head/tail at $\alpha$. 
				Using the 
				same argument as above, we see 
				that 
				the $\mu$-connecting path will be endpoint-identical to $\nu$ 
				in this case. In the directed case, note that in the case 
				$\gamma_k\notin 
				An(C)$ one can find a directed path form $\gamma_k$ to $\beta$, 
				and if 
				$\alpha$ occurs on this path one can simply choose the subpath 
				from 
				$\alpha$ to 
				$\beta$.
				
				In the case $\alpha = \beta$, analogous arguments can be made 
				by assuming 
				that $k$ is the maximal index such that there exists a 
				$\mu$-inducing path 
				from $\alpha$ to $\gamma_k$ given $C$ such that $\beta = 
				\alpha$ only 
				occurs 
				once.
			\end{proof}
			
			\begin{proof}[Proof of Propositions \ref{prop:biIP} and 
				\ref{prop:asymIP}]
				For both propositions it suffices to argue 
				that if there is a $\mu$-connecting walk in the larger graph, 
				then we 
				can also find a $\mu$-connecting walk in the smaller graph. 
				Using 
				Proposition \ref{prop:IPepiConnPath} we can find 
				endpoint-identical walks 
				that are open given $C\setminus \{\alpha\}$ and replacing 
				$\alpha 
				\starrightarrow \beta$ 
				with such a walk will give a walk which is open given $C$. For 
				Proposition 
				\ref{prop:asymIP} one should note that adding the edge respects 
				the 
				ancestry of the nodes due to transitivity.
			\end{proof}
			
			\begin{proof}[Proof of Proposition \ref{prop:noIPsep}]
				Assume there is no inducing path from $\alpha$ to $\beta$ and 
				let $\omega$ 
				be some walk from $\alpha$ to $\beta$ with a head at
				$\beta$. Note that $\omega$ must 
				have length at least 2. 
				
				$$
				\alpha = \gamma_0 \overset{e_0}{\sim} \gamma_1 
				\overset{e_1}{\sim} \ldots 
				\overset{e_{m-1}}{\sim} \gamma_m 
				\overset{e_m}{\starrightarrow} \beta.
				$$
				
				There must exist an $i \in \{0,1,\ldots,m\}$ such that 
				$\gamma_i$ is not 
				directedly collider-connected to $\beta$ along $\omega$ or such 
				that 
				$\gamma_i \notin An(\alpha, 
				\beta)$. Let $j$ be the largest such index. Note first that 
				$\gamma_m$ is 
				always directedly collider-connected to $\beta$ along $\omega$ 
				and $\gamma_0$ 
				is always in 
				$An(\alpha,\beta)$. If $j \neq m$ and $\gamma_j$ 
				is not directedly collider-connected to $\beta$ along $\omega$, 
				then 
				$\gamma_{j+1}$ is a noncollider 
				and
				$\omega$ is closed in $\gamma_{j+1} \in D(\alpha,\beta)$ (note 
				that $\alpha 
				= \gamma_{j+1}$ is impossible as there would then be an 
				inducing path from 
				$\alpha$ to $\beta$). If $j \neq 0$ and
				$\gamma_j \notin An(\alpha,\beta)$ then there is some $k \in 
				\{1,\ldots,j\}$ such that $\gamma_k$ 
				is a collider and $\gamma_k \notin An(\alpha,\beta)$ and 
				$\omega$ is 
				therefore closed in this collider.
			\end{proof}
			
			\begin{proof}[Proof of Proposition \ref{prop:SiIsPoSi}] We verify 
				that 
				(gs1)--(gs3) hold. \\
				\textbf{(gs1)} The edge $\alpha \leftrightarrow \beta$ 
				constitutes an 
				inducing path in both directions.
				
				\noindent \textbf{(gs2-3)} Let $\gamma \in V, C\subseteq V$ 
				such that 
				$\beta\in C$, and 
				assume that there is a $\mu$-connecting walk from 
				$\gamma$ to $\beta$ given $C$ in $\G$. This 
				walk has a head at $\beta$ and composing the walk with $\alpha 
				\leftrightarrow \beta$ creates an $\mu$-connecting walk from
				$\gamma$ to $\alpha$ given $C$.
			\end{proof}
			
			\begin{proof}[Proof of Lemma \ref{lem:AddPoSiME}]
				Any $\mu$-connecting walk in $\mathcal{G}$ is also present and 
				$\mu$-connecting in 
				$\mathcal{G}^+$, hence $\mathcal{I}(\mathcal{G}^+) \subseteq 
				\mathcal{I}(\mathcal{G})$.
				
				Assume $\gamma,\delta\in V, C \subseteq V$ and assume that 
				$\rho$ is a 
				$\mu$-connecting route from $\gamma$ to $\delta$ given $C$ in 
				$\mathcal{G}^+$. Let $e$ denote the edge $\alpha\leftrightarrow 
				\beta$. 
				Using (gs1), 
				there exist an inducing path from $\alpha$ to 
				$\beta$ in $\G$ and one from $\beta$ to $\alpha$. Denote these 
				by $\nu_1$ 
				and $\nu_2$. If $e$ is not in 
				$\rho$, then $\rho$ is also in $\mathcal{G}$ and 
				$\mu$-connecting as 
				the addition of the bidirected edge does not change the 
				ancestry of $\G$.
				
				If $e$ occurs twice in $\rho$ then it contains a subroute 
				$\alpha 
				\overset{e}{\leftrightarrow} \beta \overset{e}{\leftrightarrow} 
				\alpha$ and 
				$\alpha = \delta$  (or 
				with the roles 
				interchanged). Either one can find a $\mu$-connecting 
				subroute of $\rho$ with no occurrences of $e$ or $\alpha\notin 
				C$. If 
				$\beta\in C$, then compose the subroute of $\rho$ from $\gamma$ 
				to the 
				first occurrence of $\alpha$ (which is either trivial or can be 
				assumed to 
				have a tail at 
				$\alpha$) with the $\nu_1$-induced open walk from 
				$\alpha$ to $\beta$ using Proposition \ref{prop:IPepiConnPath}. 
				This is a 
				$\mu$-connecting walk in $\G$ from $\gamma$ to 
				$\beta$ and using (gs2) the result follows. If $\beta \notin 
				C$, 
				then the result follows from composing the subroute from 
				$\gamma$ to 
				$\alpha$ with the $\nu_1$-induced open walk from $\alpha$ to 
				$\beta$ and 
				the $\nu_2$-inducing open walk from 
				$\beta$ to $\alpha$.
				
				If $e$ only occurs once on 
				$\rho$, consider first a $\rho$ of the form
				
				\[
				\underbrace{\gamma \sim \ldots \sim \alpha}_{\rho_1} 
				\overset{e}{\leftrightarrow} 
				\underbrace{\beta \sim \ldots \starrightarrow \delta }_{\rho_2}.
				\]
				
				\noindent Assume first that $\alpha \notin C$. Let $\pi$ denote 
				the 
				$\nu_1$-induced open walk from $\alpha$ to $\beta$ and note 
				that $\pi$ 
				has a head at $\beta$. If $\gamma = 
				\alpha$ then $\pi$ composed with $\rho_2$ is a
				$\mu$-connecting walk from $\gamma$ to $\delta$ in $\G$. If 
				$\gamma 
				\neq 
				\alpha$ 
				we can just 
				replace $e$ with $\pi$, and the resulting composition of the
				walks $\rho_1$, $\pi$ and $\rho_2$ is a $\mu$-connecting
				walk from $\gamma$ to $\delta$ in $\G$. If instead $\alpha 
				\in C$, then $\gamma \neq \alpha$ and $\alpha$ is a collider on 
				$\rho$, 
				and
				$\rho_1$ thus has a head at $\alpha$ and is $\mu$-connecting
				from $\gamma$ to $\alpha$ given $C$ in $\G$. Using (gs3) we can 
				find a 
				$\mu$-connecting walk from $\gamma$ to 
				$\beta$ given $C$ in $\mathcal{G}$. Composing this with 
				$\rho_2$ gives a 
				$\mu$-connecting walk from 
				$\gamma$ to $\delta$ given $C$ in $\G$.
				
				If $\rho$ instead has the form
				
				\[
				\gamma \sim \ldots \sim \beta \overset{e}{\leftrightarrow} 
				\alpha \sim \ldots \starrightarrow \delta,
				\]
				
				\noindent a similar argument using (gs2) applies. In
				conclusion, $\mathcal{I}(\mathcal{G}) \subseteq
				\mathcal{I}(\mathcal{G}^+)$.  
			\end{proof}
			
			\begin{proof}[Proof of Proposition \ref{prop:PaIsPoPa}] We verify 
				that 
				(gp1)--(gp4) hold. \\
				\textbf{(gp1)} $\alpha \rightarrow \beta$ constitutes an 
				inducing 
				path from $\alpha$ to $\beta$. 
				
				\noindent \textbf{(gp2)}  Let $\omega$ be a $\mu$-connecting 
				walk from 
				$\gamma$ to 
				$\alpha$ given $C$, $\alpha\notin C$. Then $\omega$ composed 
				with 
				$\alpha\rightarrow \beta$ 
				is 
				$\mu$-connecting from $\gamma$ to $\beta$ given $C$.
				
				\noindent \textbf{(gp3)}  Let $\omega_1$ be a $\mu$-connecting 
				walk from 
				$\gamma$ to 
				$\beta$ given $C$, $\alpha\notin C, \beta\in C$, and let 
				$\omega_2$ be a 
				$\mu$-connecting 
				walk 
				from $\alpha$ to $\delta$ given $C$. The composition of 
				$\omega_1$, 
				$\alpha\rightarrow\beta$, and $\omega_2$ is $\mu$-connecting.
				
				\noindent \textbf{(gp4)}  Let $\omega$ be a $\mu$-connecting 
				walk from 
				$\beta$ to  
				$\gamma$ given $C\cup \{\alpha\}$, $\alpha\notin C$. If this 
				walk is closed 
				given $C$, 
				then there exists a collider on $\omega$, which is an ancestor 
				of $\alpha$ 
				and 
				not in $An(C)$. Let $\delta$ be the collider on $\omega$ with 
				this property 
				which is the closest to $\gamma$. Then we can find a directed 
				and open 
				path from $\delta$ to $\beta$ and composing the inverse of this 
				with the 
				subwalk of 
				$\omega$ from $\delta$ to $\gamma$ gives us a connecting walk. 
			\end{proof}
			
			\begin{proof}[Proof of Lemma \ref{lem:AddPoPaME}] As 
				\mbox{$An_{\mathcal{G}}(C) 
					\subseteq 
					An_{\mathcal{G}^+}(C)$} for all $C\subseteq V$, any 
				$\mu$-connecting 
				path 
				in $\mathcal{G}$ is also 
				$\mu$-connecting in $\mathcal{G}^+$, and it therefore follows 
				that 
				$\I(\G^+) 
				\subseteq 
				\I(\G)$. 
				
				We will prove the other inclusion by considering a 
				$\mu$-connecting walk 
				from 
				$\gamma$ to $\delta$ given $C$ in $\G^+$ and argue that we can 
				find 
				another $\mu$-connecting walk in $\G^+$ that fits into cases 
				(a) or (b) 
				below. 
				In both cases, we will use the potential parents properties 
				to argue 
				that there is also a $\mu$-connecting walk from $\gamma$ to 
				$\delta$ given 
				$C$ in $\G$. Let $e$ denote the edge $\alpha \rightarrow \beta$.
				
				Let  $\nu$ denote the inducing path 
				from $\alpha$ to $\beta$ in $\G$ which we know to exist by 
				(gp1) and 
				Proposition
				\ref{prop:noIPsep}. Say we 
				have a $\mu$-connecting walk in $\mathcal{G}^+$, $\omega$, 
				from $\gamma$ to $\delta$ given $C$. There can be two reasons 
				why $\omega$ 
				is not 
				$\mu$-connecting in 
				$\mathcal{G}$: 1) $e$ is in $\omega$, 2) there exist colliders, 
				$c_1,\ldots,c_k,$ on $\omega$, which are in 
				$An_{\mathcal{G}^+}(C)$ but not 
				in $An_{\mathcal{G}}(C)$. We will in this proof call such 
				colliders {\it 
					newly closed}. 	If there exists a newly closed collider on 
				$\omega$, 
				$c_i$, 
				then there exists in $\mathcal{G}$ a directed path from $c_i$ 
				to $\alpha$ on which no node is in $C$, and furthermore 
				$\alpha\notin C$. 
				Note that this path does not contain $\beta$, and the existence 
				of a 
				newly closed collider implies that $\beta \in 
				An_{\mathcal{G}}(C)$.
				
				Using Proposition 
				\ref{prop:muConnPath}, we can find a route, $\rho$, in $\G^+$ 
				from $\gamma$ 
				to 
				$\delta$, which is $\mu$-connecting in $\G^+$. Assume first 
				that $e$ occurs 
				at most 
				once on $\rho$. 
				If there are newly closed colliders on $\rho$, we will argue 
				that we can 
				find a $\mu$-connecting walk in $\G^+$ with no newly closed 
				colliders and 
				such 
				that $e$ occurs at most once. 
				Assume that 
				$c_1,\ldots,c_k$ are newly closed 
				colliders, ordered by their occurrences on the route $\rho$. We 
				allow 
				for 
				$k=1$, in which case $c_1 = c_k$. We will divide the argument
				into three cases, and we use 
				in all three cases that a $\mu$-connecting walk in $\G$ is also 
				present in 
				$\G^+$ and has no newly closed colliders nor occurrences of 
				$e$. We also 
				use that $\alpha \notin C$ when applying (gp2).
				
				\begin{enumerate}[label=(\roman*)]
					\item $e$ is between $\gamma$ and $c_1$ on $\rho$. \newline 
					Consider the subwalk of $\rho$ from $\gamma$ to the first 
					occurrence of 
					$\alpha$. If this subwalk has a tail at $\alpha$ (or is 
					trivial) then 
					we can compose 
					it with 
					the inverse 
					of the path from $c_k$ to $\alpha$ and the subwalk from 
					$c_k$ to 
					$\delta$. This walk is open. If there is a head at 
					$\alpha$, then using (gp2) we can find a $\mu$-connecting 
					walk 
					from $\gamma$ to $\beta$ in $\G$, compose it with $e$, the 
					inverse of 
					the path 
					from $c_k$ to $\alpha$ and the subwalk from $c_k$ to 
					$\delta$. This is 
					open as $\beta\in An_\G(C)$ and $\alpha \notin C$ whenever 
					there exist 
					newly closed 
					colliders.
					\item $e$ is between $c_k$ and $\delta$ on $\rho$. \newline 
					Consider the subwalk of $\rho$ from $\gamma$ to $c_1$, and 
					compose it 
					with the directed path from $c_1$ to $\alpha$. This is 
					$\mu$-connecting 
					in $\G$ and using (gp2) we can find a $\mu$-connecting walk 
					in $\G$ 
					from $\gamma$ to $\beta$. Composing this walk with the 
					subwalk of 
					$\rho$ from 
					$\beta$ to 
					$\delta$ gives a $\mu$-connecting walk from $\gamma$ to 
					$\delta$, 
					noting that $\beta \in An_\G(C)$. 
					\item $e$ is between $c_1$ and $c_k$ on $\rho$ or not on 
					$\rho$ at all. 
					\newline
					Composing the subwalk from $\gamma$ to $c_1$ with the
					directed path from $c_1$ to $\alpha$ gives a 
					$\mu$-connecting 
					walk from $\gamma$ to $\alpha$ given
					$C$ in $\mathcal{G}$, and by (gp2) 
					we 
					can 
					find a $\mu$-connecting walk from $\gamma$ to $\beta$ in 
					$\G$, thus 
					there 
					are no newly closed colliders on this walk and it does not 
					contain $e$. 
					Composing it with $e$, the directed path from $c_k$ to 
					$\alpha$ and the 
					subwalk from $c_k$ to $\delta$ gives a $\mu$-connecting 
					walk in $\G^+$.
				\end{enumerate}
				
				\noindent In all cases (i), (ii), and (iii) we have argued that 
				there 
				exists 
				a $\mu$-connecting 
				walk from $\gamma$ to $\delta$ in $\G^+$ 
				that contains no 
				newly closed colliders and that contains $e$ at most once.
				Denote 
				this walk by $\tilde{\omega}$. If $\tilde{\omega}$ does not 
				contain $e$ at 
				all, then we are done. Otherwise, two cases 
				remain, 
				depending on the orientation of $e$ in the 
				$\mu$-connecting walk $\tilde{\omega}$:
				
				\begin{enumerate}[label=(\alph*)]
					\item Assume first we have a walk
					of the form
					
					$$
					\gamma \sim \ldots \overset{e_\alpha}{\sim} \alpha 
					\rightarrow \beta 
					\sim 
					\ldots \starrightarrow \delta,
					$$
					
					\noindent If there is a tail on $e_\alpha$ at 
					$\alpha$, or if $\gamma=\alpha$, then we can 
					substitute $e$ with the open path between $\alpha$ and 
					$\beta$ induced 
					by 
					$\nu$ and obtain an open walk. Otherwise, assume a head on 
					$e_\alpha$ 
					at $\alpha$. $\tilde{\omega}$ is $\mu$-connecting in $\G^+$ 
					and 
					therefore $\alpha \notin C$. Using (gp2), there exists a 
					$\mu$-connecting 
					walk 
					from $\gamma$ to $\beta$, and composing this walk with the 
					(potentially 
					trivial) 
					subwalk from $\beta$ to $\delta$ gives a
					$\mu$-connecting walk from $\gamma$ to $\delta$ given
					$C$ in $\mathcal{G}$.
					
					\item Consider instead a walk of the form
					
					$$
					\gamma \sim \ldots \overset{e_\beta}{\sim} \beta 
					\leftarrow 
					\alpha 
					\sim \ldots \starrightarrow \delta.
					$$
					
					\noindent If there is a head on $e_\beta$ at $\beta$,
					$\beta$ is a collider. If $\beta \in C$,	then (gp3)
					directly gives a $\mu$-connecting walk from $\gamma$
					to $\delta$ given $C$ in $\mathcal{G}$. If instead $\beta 
					\in 
					An_{\G^+}(C)\setminus C$ then we can 
					find a 
					directed path, $\pi$, in $\G^+$ from $\beta$ to 
					$\varepsilon \in C$. 
					The 
					edge $e$ 
					is not present on $\pi$ and therefore we can compose the 
					subwalk from 
					$\gamma$ to $\beta$ with $\pi$, 
					$\pi^{-1}$, and the subwalk from $\beta$ to $\delta$ to 
					obtain an open 
					walk 
					from $\gamma$ to $\delta$ without any newly closed 
					colliders, only 
					one 
					occurrence of $e$, and such that there is a tail at
					$\beta$ just before the occurence of $e$. 
					
					We have reduced this case to walks, $\tilde{\omega}$, of 
					the form
					$$
					\underbrace{\gamma \sim \ldots \leftarrow 
						\beta}_{\tilde{\omega}_1}
					\leftarrow 
					\underbrace{\alpha 
						\sim \ldots \starrightarrow \delta}_{\tilde{\omega}_2},
					$$
					where $\tilde{\omega}_1$ is potentially trivial. Let 
					$\bar{\pi}$ denote 
					the $\nu$-induced open path or cycle 
					from $\alpha$ to $\beta$ in $\G$. Using Proposition
					\ref{prop:muConnPath} there is a $\mu$-connecting
					route, $\bar{\rho}$, 
					from $\alpha$ to $\delta$ given $C$ in $\G$. 
					%We can find a 
					%$\mu$-connecting route 
					%from $\alpha$ to $\delta$ in $\G$, $\bar{\rho}$, by using 
					%Proposition 
					%\ref{prop:muConnPath}. 
					If there is a tail at 
					$\alpha$ on 
					$\bar{\rho}$ 
					or on $\bar{\pi}$ then the composition of
					$\tilde{\omega}_1$, $\overline{\pi}$ and
					$\bar{\rho}$ is 
					$\mu$-connecting. Otherwise, if $\alpha \neq \beta$,
					the composition of $\overline{\pi}$ and
					$\bar{\rho}$ is a $\mu$-connecting walk from $\beta$ to 
					$\delta$ given 
					$C\cup 
					\{\alpha\}$ in $\G$ as $\alpha$ does not occur as a 
					noncollider on 
					this composition. Using 
					(gp4) there is also one given $C$. As 
					there is a tail at $\beta$ on $\tilde{\omega}$ we can 
					compose
					$\tilde{\omega}_1$ with this walk
					to 
					obtain 
					an open walk from $\gamma$ to $\delta$ given $C$ in 
					$\mathcal{G}$. If 
					$\alpha = \beta$ 
					the composition of $\tilde{\omega}_1$ with 
					$\tilde{\omega}_2$ is an 
					open walk from $\gamma$ to $\delta$ given $C$ in 
					$\mathcal{G}$.
				\end{enumerate}
				
				Assume finally that $e$ occurs 
				twice on $\rho$. In this case $\rho$ contains a subroute $\beta 
				\overset{e}{\leftarrow} 
				\alpha \overset{e}{\rightarrow} \beta$ and $\beta = \delta$. In 
				this case 
				$\alpha \notin 
				C$. If there are any newly closed colliders, consider the one 
				closest to 
				$\gamma$, $c$. The subroute of $\rho$ from $\gamma$ to $c$ 
				composed with 
				the directed path from $c$ to $\alpha$ gives a $\mu$-connecting 
				path and 
				(gp2) gives the result. Else if there is a head at $\alpha$ on 
				the 
				$\nu$-induced open walk then (gp2) again gives the result. 
				Otherwise, 
				compose the subroute from $\gamma$ to the first $\beta$, the 
				inverse of the 
				$\nu$-induced open walk, and the $\nu$-induced open walk to 
				obtain an open 
				walk in $\G$ from $\gamma$ to $\beta = \delta$.
			\end{proof}
			
			\begin{proof}[Proof of Theorem \ref{thm:maxElm}]
				Propositions 
				\ref{prop:SiIsPoSi} and \ref{prop:PaIsPoPa} show that 
				$\mathcal{N}$ is in 
				fact a supergraph of $\mathcal{G}$, and as $E^m$ only depends 
				on the 
				independence model, it also shows that $\mathcal{N}$ is a 
				supergraph of any 
				element in $[\mathcal{G}]$. We can sequentially add the edges 
				that are in 
				$\mathcal{N}$ but not in $\G$, and Lemmas \ref{lem:AddPoSiME} 
				and 
				\ref{lem:AddPoPaME} show that this is done Markov equivalently, 
				meaning 
				that $\mathcal{N} \in [\mathcal{G}]$.
			\end{proof}
			
			\begin{lem}
				Let $\alpha,\beta\in V$. If there is a directed edge, $e$, from 
				$\alpha$ to $\beta$, and a unidirected inducing path from 
				$\alpha$ to 
				$\beta$ of length at least two in $\mathcal{N}$, then there is 
				a directed 
				inducing path from 
				$\alpha$ to $\beta$ in $\mathcal{N} - e$.
				\label{lem:maxNondirToDir}
			\end{lem}
			
			\begin{proof}[Proof of Lemma \ref{lem:maxNondirToDir}]
				Let $\nu$ denote the unidirected inducing path and 
				$\gamma_1,\ldots,\gamma_n$ the non-endpoint nodes of $\nu$. 
				Then 
				$\gamma_i \in An_\mathcal{N}(\{\alpha,\beta\})$ and also 
				$\gamma_i 
				\in 
				An_\mathcal{N}(\beta)$ due to the directed edge from $\alpha$
				to $\beta$. It follows that either $\gamma_i \in 
				An_\mathcal{N}(\alpha)$ 
				or $\gamma_i \in An_{(\mathcal{N} - e)}(\beta)$. If $\gamma_i 
				\in 
				An_\mathcal{N}(\alpha)$, let $e_i$ denote the directed edge 
				from $\gamma_i$ 
				to  
				$\beta$, and let $\mathcal{N}^+ = (V, F \cup \{e_i\})$. We will 
				argue that 
				$\mathcal{N} = \mathcal{N}^+$ using the 
				maximality of $\mathcal{N}$. Note first that the edge does not 
				change the 
				ancestry of the 
				graph in the sense that $An_\mathcal{N}(\gamma) = 
				An_{\mathcal{N}^+}(\gamma)$ for all $\gamma\in V$. Note also 
				that there is a
				bidirected inducing path between $\gamma_i$ and $\beta$ in 
				$\mathcal{N}$, 
				and therefore $\gamma_i \leftrightarrow_{\mathcal{N}} \beta$. 
				Assume 
				that 
				$e_i$ is in a $\mu$-connecting path in $\mathcal{N}^+$. There 
				is a directed 
				path from $\gamma_i$ to $\alpha$ in $\mathcal{N}$ and therefore 
				$e_i$ can 
				either be 
				substituted with $\gamma_i \rightarrow \alpha_i \rightarrow 
				\ldots 
				\rightarrow \alpha_k \rightarrow \alpha \rightarrow \beta$ (if 
				$\alpha_1,\ldots,\alpha_k,\alpha\notin C$), or with $\gamma_i 
				\leftrightarrow \beta$ (otherwise), and we see that 
				$\mathcal{I}(\mathcal{N}) = \mathcal{I}(\mathcal{N}^+)$. By 
				maximality of 
				$\mathcal{N}$ we have that 
				$\mathcal{N} = \mathcal{N}^+$ which implies that $e_i \in F$. 
				Thus 
				$\gamma_i \in 
				An_{(\mathcal{N} - e)}(\beta)$. This shows that $\nu$ is 
				also a 
				directed inducing path in $\mathcal{N} - e$.
			\end{proof}
			
			\begin{lem}
				Let edges $\alpha \rightarrow 
				\beta$, 
				$\beta \rightarrow \alpha$ and $\alpha \leftrightarrow \beta$ 
				be denoted by 
				$e_1, e_2, e_3$, respectively. If $e_1, e_3 \in F$, then 
				$\mathcal{N} - e_1 
				\in 
				[\mathcal{N}]$. If $e_1,e_2,e_3 \in F$, then $\mathcal{N} - e_3 
				\in 
				[\mathcal{N}]$.
				\label{lem:maxRemove}
			\end{lem}
			
			\begin{proof}[Proof of Lemma \ref{lem:maxRemove}]
				Note that if edges $\gamma \starrightarrow 
				\alpha$, $\alpha \leftrightarrow \beta$, and $\alpha 
				\rightarrow \beta$ 
				are present in a maximal DMG, then so is $\gamma 
				\starrightarrow \beta$ 
				by Propositions \ref{prop:biIP} and \ref{prop:asymIP}. Assume 
				$e_1,e_3\in 
				E$. Using the above observation, note that every 
				vertex that is a parent of 
				$\alpha$ in 
				$\mathcal{N}$ is also 
				a parent of $\beta$, thus $An_\mathcal{N}(\delta) \setminus 
				\{\alpha\} = 
				An_{(\mathcal{N} - 
					e_1)}(\delta) \setminus \{\alpha\}$ for all $\delta \in V$. 
				Consider a 
				$\mu$-connecting 
				walk, $\omega$, in $\mathcal{N}$ given $C$. Any collider 
				different from 
				$\alpha$ on this walk is in $An_{(\mathcal{N} - e_1)}(C)$. If 
				$\alpha 
				\notin An_{(\mathcal{N} - 
					e_1)}(C) $ is a 
				collider, 
				then we can substitute the subwalk $\gamma_1 \starrightarrow 
				\alpha 
				\starleftarrow \gamma_2$ with $\gamma_1 \starrightarrow \beta 
				\starleftarrow \gamma_2$. If $e_1$ is the first edge on 
				$\omega$ and $\alpha$ the first node, then 
				just substitute 
				$e_1$ with $e_3$. Else, we need to consider two cases: in the 
				first case 
				there is a 
				subwalk 
				$\gamma 
				\starrightarrow \alpha \rightarrow \beta$ (or $\beta 
				\leftarrow \alpha \starleftarrow \gamma$) and therefore an edge 
				$\gamma 
				\starrightarrow \beta$ in $\mathcal{N} - e_1$ if $\gamma \neq 
				\alpha$. If 
				$\gamma = \alpha$, we can 
				simply remove the loop, replacing $e_1$ with $e_3$ if $\gamma$ 
				was the 
				final node on $\omega$. In the second 
				case, there is a subwalk 
				$\gamma 
				\leftarrow \alpha \rightarrow \beta$ (or $\beta 
				\leftarrow \alpha \rightarrow \gamma$), and we can substitute 
				$e_1$ with 
				$e_3$ if $\beta \neq \gamma$. If $\beta 
				= \gamma$, then we can substitute 
				$\beta \leftarrow \alpha \rightarrow \beta$ with $\beta 
				\leftrightarrow 
				\beta$.
				
				The proof of the other statement is
				similar.
			\end{proof}

			\begin{proof}[Proof of Proposition \ref{prop:dirEdgeSepInN}]
				One implication is immediate by contraposition: if $\alpha 
				\notin 
				u(\beta, \I(\mathcal{N} - e))$, then $\mathcal{N} - e \notin 
				[\mathcal{N}]$. 
				
				Assume $\alpha \in u(\beta, \mathcal{I}(\mathcal{N} - e))$. 
				There 
				exists an inducing path, $\nu$, from $\alpha$ to $\beta$ in 
				$\mathcal{N} - e$. If $\nu$ is directed, then the conclusion 
				follows from 
				Proposition \ref{prop:asymIP}. If $\nu$ is unidirected and of 
				length one, 
				then it is also directed. If it is unidirected and has length 
				at least two, 
				it follows from 
				Lemma \ref{lem:maxNondirToDir} that there also exists a 
				directed inducing path 
				in $\mathcal{N} - e$. Proposition \ref{prop:asymIP} finishes 
				the argument. 
				Assume that $\nu$ is bidirected. Then $\alpha 
				\leftrightarrow_\mathcal{N} \beta$ due to maximality and 
				Proposition 
				\ref{prop:biIP}. Lemma \ref{lem:maxRemove} gives the result.
			\end{proof}
			
			\begin{proof}[Proof of Proposition \ref{prop:bidirEdgeSepInN}]
				One implication follows by contraposition. Assume instead that 
				$\alpha \in 
				u(\beta, \I(\mathcal{N} - e))$ 
				and $\beta \in u(\alpha, \I(\mathcal{N} - e))$. Then there is 
				an inducing 
				path from $\alpha$ to $\beta$ and one from $\beta$ to $\alpha$ 
				in 
				$\mathcal{N} - e$. Denote these by $\nu_1$ and $\nu_2$. If one 
				of them is 
				bidirected, then the conclusion follows. Assume instead that 
				none of them 
				are bidirected and assume first that both are a single edge. 
				The conclusion 
				then follows using Lemma \ref{lem:maxRemove}.
				
				Assume now that $\nu_1$ or $\nu_2$ is an inducing path of 
				length at 
				least 2. Say that $\beta \rightarrow \gamma_1 \leftrightarrow 
				\ldots 
				\leftrightarrow \gamma_m 
				\leftrightarrow \alpha$ is an inducing path. If $\nu_1$ is the
				inducing path $\alpha 
				\rightarrow_\mathcal{N} \beta$ of length one, then there is 
				also a 
				bidirected inducing 
				path 
				between $\gamma_1$ and $\beta$ in $\mathcal{N}$, and there will 
				also be a 
				bidirected inducing path in $\mathcal{N} - e$ between $\alpha$ 
				and $\beta$. 
				If 
				instead $\nu_1$ is the 
				inducing path $\alpha \rightarrow \phi_1 \leftrightarrow 
				\ldots  
				\leftrightarrow \phi_k \leftrightarrow 
				\beta$ then $\gamma_1 \leftrightarrow_\mathcal{N} \phi_1$. In 
				this case 
				$\alpha \leftrightarrow \gamma_m \ldots 
				\gamma_1 \leftrightarrow \phi_1 
				\ldots \phi_k \leftrightarrow \beta$ can be trimmed down to a 
				bidirected 
				inducing path in 
				$\mathcal{N} - e$. 
			\end{proof}

%%%%%%%%%%%%%%%%%%%%%%%%%%%%%%%%%%%%
%%%%%%%%%%%%%%%%%%%%%%%%%%%%%%%%%%%% biblio

\appendix

\bibliography{paperRef}

\end{document}